\documentclass[12pt]{amsart}

\usepackage{
  amsmath,
  amsfonts,
  amssymb,
  amsthm,
  amscd,
  comment,
  paralist,
  etoolbox,
  mathtools,
  enumitem,
  mathdots,
}
\usepackage[usenames,dvipsnames]{xcolor}
\usepackage[all]{xy}

\usepackage{bbm}                     
\usepackage[colorlinks=true, linkcolor=blue, citecolor=blue, urlcolor=blue, breaklinks=true]{hyperref}

\DeclareFontFamily{OT1}{pzc}{}
\DeclareFontShape{OT1}{pzc}{m}{it}{<-> s * [1.10] pzcmi7t}{}
\DeclareMathAlphabet{\mathpzc}{OT1}{pzc}{m}{it}

\usepackage[capitalize]{cleveref}   

\crefname{defin}{Definition}{Definitions}
\crefname{eg}{Example}{Examples}
\crefname{lem}{Lemma}{Lemmas}
\crefname{theo}{Theorem}{Theorems}
\crefname{equation}{}{}
\crefname{enumi}{}{}

\newcommand\N{\mathbb{N}}
\newcommand\Q{\mathbb{Q}}

\newcommand\Z{\mathbb{Z}}
\newcommand\kk{\Bbbk}
\newcommand\one{\mathbbm{1}}

\newcommand\cC{\mathcal{C}}

\newcommand\cI{\mathcal{I}}

\newcommand\cY{\mathcal{Y}}

\newcommand\fh{\mathfrak{h}}

\newcommand{\md}{\textup{-mod}}
\newcommand{\bmd}{\textup{-bimod}}

\newcommand\Heis{\mathpzc{Heis}}                        
\newcommand\Par{\mathpzc{Par}}                          
\newcommand\Peis[1][]{{\Heis_{\uparrow \downarrow}^{#1}}}
\newcommand\Rep{{\mathrm{\underline{Re}p}}}
\newcommand\rHeis{\mathrm{Heis}}
\newcommand\KDelt{\Delta_{\mathrm{Kr}}}   
\newcommand\parcop{\delta}                

\DeclareRobustCommand{\stirling}{\genfrac\{\}{0pt}{}}

\DeclareMathOperator{\End}{End}

\DeclareMathOperator{\Hom}{Hom}

\DeclareMathOperator{\Add}{Add}
\DeclareMathOperator{\Ind}{Ind}

\DeclareMathOperator{\Res}{Res}
\DeclareMathOperator{\Sym}{Sym}
\DeclareMathOperator{\Tr}{Tr}

\DeclareMathOperator{\Kar}{Kar}

\usepackage{tikz}
\usetikzlibrary{arrows,patterns}
\usepackage{tikz-cd}

\newcommand{\pd}[1]{\filldraw[black] (#1) circle (1.5pt)} 

\newcommand{\braidto}{to[out=up,in=down]}

\tikzset{anchorbase/.style={>=To,baseline={([yshift=-0.5ex]current bounding box.center)}}}

\newtheorem{theo}{Theorem}[section]

\newtheorem{prop}[theo]{Proposition}
\newtheorem{lem}[theo]{Lemma}
\newtheorem{cor}[theo]{Corollary}

\theoremstyle{definition}

\newtheorem{rem}[theo]{Remark}

\numberwithin{equation}{section}
\allowdisplaybreaks

\setenumerate[1]{label=(\alph*)}          

\newtoggle{comments}
\newtoggle{details}
\newtoggle{detailsnote}

\iftoggle{comments}{%
  \newcommand{\acomments}[1]{
    \ \\
    {\color{red}
      \textbf{AS:} #1
    }
    \ \\
    }
  \newcommand{\scomments}[1]{
    \ \\
    {\color{red}
      \textbf{SNL:} #1
    }
    \ \\
    }
}{%
  \newcommand{\acomments}[1]{}
  \newcommand{\scomments}[1]{}
}

\iftoggle{details}{%
  \newcommand{\details}[1]{
      \ \\
      {\color{OliveGreen}
        \textbf{Details:} #1
      }
      \\
  }
}{%
  \newcommand{\details}[1]{}
}

\begin{document}
%

\title{Embedding Deligne's category $\Rep(S_t)$ in the Heisenberg category}

\author[Samuel Nybobe Likeng and Alistair Savage]{Samuel Nyobe Likeng and Alistair Savage \\ (appendix with Christopher Ryba)}
\address[S.N.L]{
  Department of Mathematics and Statistics \\
  University of Ottawa \\
  Ottawa, ON K1N 6N5, Canada
}
\email{snyob030@uottawa.ca}

\address[A.S.]{
  Department of Mathematics and Statistics \\
  University of Ottawa \\
  Ottawa, ON K1N 6N5, Canada
}
\urladdr{\href{https://alistairsavage.ca}{alistairsavage.ca}, \textrm{\textit{ORCiD}:} \href{https://orcid.org/0000-0002-2859-0239}{orcid.org/0000-0002-2859-0239}}
\email{alistair.savage@uottawa.ca}

\address[C.R.]{
  Department of Mathematics \\
  Massachusetts Institute of Technology \\
  Cambridge, MA 02139-4307, USA
}
\email{ryba@mit.edu}

\begin{abstract}
  We define a faithful linear monoidal functor from the partition category, and hence from Deligne's category $\mathrm{\underline{Re}p}(S_t)$, to the additive Karoubi envelope of the Heisenberg category.   We show that the induced map on Grothendieck rings is injective and corresponds to the Kronecker coproduct on symmetric functions.
\end{abstract}

\subjclass[2020]{Primary 18M05; Secondary 20C30, 17B10}

\keywords{Categorification, partition category, Deligne category, Heisenberg category, symmetric group, partition algebra, monoidal category}

\ifboolexpr{togl{comments} or togl{details}}{%
  {\color{magenta}DETAILS OR COMMENTS ON}
}{%
}

\maketitle
\thispagestyle{empty}

\section{Introduction\label{intro}}

In \cite{Del07}, Deligne introduced a linear monoidal category $\Rep(S_t)$ that interpolates between the categories of representations of the symmetric groups.  In particular, when $t$ is a nonnegative integer $n$, the category of representations of $S_n$ is equivalent to the quotient of $\Rep(S_n)$ by the tensor ideal of negligible morphisms.  One particularly efficient construction of $\Rep(S_t)$ is as the additive Karoubi envelope of the \emph{partition category} $\Par(t)$. The endomorphism algebras of the partition category are the \emph{partition algebras} first introduced by Martin (\cite{Mar94}) and later, independently, by Jones (\cite{Jon94}) as a generalization of the Temperley--Lieb algebra and the Potts model in statistical mechanics.  The partition algebras are in duality with the action of the symmetric group on tensor powers of its permutation representation; that is, the partition algebras generate the commutant of this action  (see \cite[Th.~3.6]{HR05} and \cite[Th.~8.3.13]{CST10}).

In \cite{Kho14}, Khovanov defined another linear monoidal category, the \emph{Heisenberg category} $\Heis$, which is also motivated by the representation theory of the symmetric groups.  In particular, $\Heis$ acts on $\bigoplus_{n \ge 0} S_n\md$, where its two generating objects act by induction $S_n\md \to S_{n+1}\md$ and restriction $S_{n+1}\md \to S_n\md$.  Morphisms in $\Heis$ act by natural transformations between compositions of induction and restriction functors.

Deligne's category $\Rep(S_t)$ can be thought of as describing the representation theory of $S_n$ for arbitrary $n$ in a uniform way, but with $n$ fixed (and not necessarily a nonnegative integer).  On the other hand, the Heisenberg category goes further, allowing $n$ to vary and describing the representation theory of all the symmetric groups at once.  Thus, it is natural to expect a precise relationship between the two categories, with the Heisenberg category being larger.  The goal of the current paper is to describe such a relationship.

For the purposes of this introduction, we describe our results in the case where $t$ is generic. Our first main result (\cref{functordef,faithful,bike}) is the construction of a faithful strict linear monoidal functor
\[
  \Psi_t \colon \Par(t) \to \Heis.
\]
This functor sends $t$ to the clockwise bubble in $\Heis$ and is compatible with the actions of $\Par(t)$ and $\Heis$ on categories of modules for symmetric groups (\cref{actcom}).  Since Deligne's category $\Rep(S_t)$ is the additive Karoubi envelope of the partition category, we have an induced faithful linear monoidal functor
\[
  \Psi_t \colon \Rep(S_t) \to \Kar(\Heis),
\]
where $\Kar(\Heis)$ denotes the additive Karoubi envelope of the Heisenberg category $\Heis$.

The Grothendieck ring of $\Rep(S_t)$ is isomorphic to the ring $\Sym$ of symmetric functions.  On the other hand, the Grothendieck ring of $\Heis$ is isomorphic to a central reduction $\rHeis$ of the universal enveloping algebra of the Heisenberg Lie algebra.  This was conjectured by Khovanov in \cite[Conj.~1]{Kho14} and recently proved in \cite[Th.~1.1]{BSW18}.  We thus have an induced map
\[
  [\Psi_t] \colon \Sym \cong K_0(\Rep(S_t)) \to K_0(\Heis) \cong \rHeis.
\]
Our second main result (\cref{finally}) is that this map is injective and is given by the Kronecker coproduct on $\Sym$.  We also describe the map induced by $\Psi_t$ on the traces (or zeroth Hochschild homologies) of $\Rep(S_t)$ and $\Heis$.

The partition algebras contain many so-called \emph{diagram algebras} that have been well-studied in the literature.  These include the Brauer algebras, Temperley--Lieb algebras, rook algebras, planar partition algebras, planar rook algebras, rook-Brauer algebras, and Motzkin algebras.  As a result, the functor $\Psi_t$ also yields explicit embeddings of these algebras into endomorphism rings in the Heisenberg category, and of their associated categories into the Heisenberg category.

We expect that the results of this paper are the starting point of a large number of precise connections between various algebras and categories that are well-studied in the literature.  We list here some such possible extensions of the current work:
\begin{enumerate}
  \item Replacing the role of the symmetric group with wreath product algebras, one should be able to define an embedding, analogous to $\Psi_t$, relating the $G$-colored partition algebras of \cite{Blo03}, the wreath Deligne categories of \cite{Mor12,Kno07}, and the Frobenius Heisenberg categories of \cite{RS17,Sav18Frob}.

  \item Quantum versions of $\Psi_t$ should exist relating the $q$-partition algebras of \cite{HT10}, a quantum analogue of Deligne's category, and the quantum Heisenberg category of \cite{LS13,BSW18quant}.

  \item Replacing the role of the symmetric group by more general degenerate cyclotomic Hecke algebras should relate the categories of \cite[\S5.1]{Eti14} to the higher central charge Heisenberg categories of \cite{MS18,Bru18}.
\end{enumerate}

The organization of this paper is as follows.  In \cref{sec:partition} we recall the definition of the partition category and Deligne's category $\Rep(S_t)$.  We then recall the Heisenberg category in \cref{sec:Heis}.  We define the functor $\Psi_t$ in \cref{sec:functor}.  In \cref{sec:faithful} we show that $\Psi_t$ intertwines the natural categorical actions on categories on modules of symmetric groups and that it is faithful when $\kk$ is an integral domain of characteristic zero.  Finally, in \cref{sec:Groth} we discuss the induced map on Grothendieck rings and traces.  In \cref{appendix}, we show that $\Psi_t$ is faithful when $\kk$ is any commutative ring.

\subsection*{Notation}

Throughout, we work over a ground ring $\kk$, which is an arbitrary commutative ring unless otherwise specified.  We let $\N$ denote the additive monoid of nonnegative integers.

\subsection*{Acknowledgements}

This research of A.~Savage was supported by Discovery Grant RGPIN-2017-03854 from the Natural Sciences and Engineering Research Council of Canada.  S.~Nyobe Likeng was also supported by this Discovery Grant.  The authors would like to thank Georgia Benkart, Victor Ostrik, Michael Reeks, and Ben Webster for useful conversations, Jon Brundan for helpful comments on an earlier draft of the paper, and Christopher Ryba for suggesting the proof given in \cref{appendix}.

\section{The partition category and Deligne's category $\Rep(S_t)$\label{sec:partition}}

In this section we recall the definition and some important facts about one of our main objects of study.  We refer the reader to \cite{Lau12,Sav18} for a brief treatment of the language of string diagrams and strict linear monoidal categories suited to the current work.  For a morphism $X$ in a category, we will denote the identity morphism on $X$ by $1_X$.

For $m,\ell \in \N$, a \emph{partition} of type $\binom{\ell}{m}$ is a partition of the set $\{1,\dotsc,m,1',\dotsc,\ell'\}$.  The elements of the partition will be called \emph{blocks}.  We depict such a partition as a graph with $\ell$ vertices in the top row, labelled $1',\dotsc,\ell'$ from \emph{right to left}, and $m$ vertices in the bottom row, labelled $1,\dotsc,m$ from \emph{right to left}.  (We choose the right-to-left numbering convention to better match with the Heisenberg category later.)  We draw edges so that the blocks are the connected components of the graph.  For example, the partition $\big\{ \{1,5\}, \{2\}, \{3,1'\}, \{4,4',7'\}, \{2', 3'\}, \{5'\}, \{6'\} \big\}$ of type $\binom{7}{5}$ is depicted as follows:
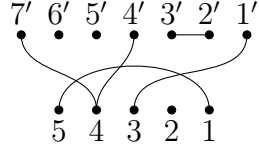
\[
  \begin{tikzpicture}[anchorbase]
    \pd{0.5,0} node[anchor=north] {$5$};
    \pd{1,0} node[anchor=north] {$4$};
    \pd{1.5,0} node[anchor=north] {$3$};
    \pd{2,0} node[anchor=north] {$2$};
    \pd{2.5,0} node[anchor=north] {$1$};
    \pd{0,1} node[anchor=south] {$7'$};
    \pd{0.5,1} node[anchor=south] {$6'$};
    \pd{1,1} node[anchor=south] {$5'$};
    \pd{1.5,1} node[anchor=south] {$4'$};
    \pd{2,1} node[anchor=south] {$3'$};
    \pd{2.5,1} node[anchor=south] {$2'$};
    \pd{3,1} node[anchor=south] {$1'$};
    \draw (1,0) \braidto (0,1);
    \draw (1,0) \braidto (1.5,1);
    \draw (0.5,0) to[out=up,in=up] (2.5,0);
    \draw (1.5,0) \braidto (3,1);
    \draw (2,1) to (2.5,1);
  \end{tikzpicture}
\]
Note that different graphs can correspond to the same partition since only the connected components of the graph are relevant.

From now on, we will omit the labels of the vertices when drawing partition diagrams.  We write $D \colon m \to \ell$ to indicate that $D$ is a partition of type $\binom{\ell}{m}$.  We denote the unique partition diagrams of types $\binom{1}{0}$ and $\binom{0}{1}$ by
\[
  \begin{tikzpicture}[anchorbase]
    \pd{0,0.5};
    \draw (0,0.25) to (0,0.5);
  \end{tikzpicture}
  \ \colon 0 \to 1
  \qquad \text{and} \qquad
  \begin{tikzpicture}[anchorbase]
    \pd{0,0};
    \draw (0,0) to (0,0.25);
  \end{tikzpicture}
  \ \colon 1 \to 0.
\]

Given two partitions $D' \colon m \to \ell$, $D \colon \ell \to k$, one can stack $D$ on top of $D'$ to obtain a diagram $\begin{matrix} D \\ D' \end{matrix}$ with three rows of vertices.  We let $\alpha(D, D')$ denote the number of components containing only vertices in the middle row of $\begin{matrix} D \\ D' \end{matrix}$.  Let $D \star D'$ be the partition of type $\binom{k}{m}$ with the following property: vertices are in the same block of $D \star D'$ if and only if the corresponding vertices in the top and bottom rows of $\begin{matrix} D \\ D' \end{matrix}$ are in the same block.

Recall that $\kk$ is a commutative ring and fix $t \in \kk$.  The \emph{partition category} $\Par(t)$ is the strict $\kk$-linear monoidal category whose objects are nonnegative integers and, given two objects $m,\ell$ in $\Par(t)$, the morphisms from $m$ to $\ell$ are $\kk$-linear combinations of partitions of type $\binom{\ell}{m}$.  The vertical composition is given by
\[
  D \circ D'
  =t^{\alpha(D,D')} D \star D'
\]
for composable partition diagrams $D,D'$, and extended by linearity.  The bifunctor $\otimes$ is given on objects by
\[
  \otimes \colon \Par(t) \times \Par(t) \to \Par(t),\quad (m,n)\mapsto m+n.
\]
The tensor product on morphisms is given by horizontal juxtaposition of diagrams, extended by linearity.

For example, if
\[
  D' =
  \begin{tikzpicture}[anchorbase]
    \pd{0,1};
    \pd{0.5,1};
    \pd{1,1};
    \pd{1.5,1};
    \pd{2,1};
    \pd{2.5,1};
    \pd{3,1};
    \pd{0.5,0};
    \pd{1,0};
    \pd{1.5,0};
    \pd{2,0};
    \pd{2.5,0};
    \draw (1,0) \braidto (0,1);
    \draw (1,0) \braidto (1.5,1);
    \draw (0.5,0) to[out=up,in=up] (2.5,0);
    \draw (1.5,0) \braidto (3,1);
    \draw (2,1) to (2.5,1);
  \end{tikzpicture}
  \qquad \text{and}  \qquad
  D =
  \begin{tikzpicture}[anchorbase]
    \pd{0,1};
    \pd{0.5,1};
    \pd{1,1};
    \pd{1.5,1};
    \pd{2,1};
    \pd{-0.5,0};
    \pd{0,0};
    \pd{0.5,0};
    \pd{1,0};
    \pd{1.5,0};
    \pd{2,0};
    \pd{2.5,0};
    \draw (-0.5,0) \braidto (0,1);
    \draw (0.5,1) to[out=down,in=down] (1.5,1);
    \draw (0,0) to[out=up,in=up] (1,0);
    \draw (2.5,0) \braidto (1,1);
    \draw (2.5,0) \braidto (2,1);
  \end{tikzpicture}
\]
then
\[
  \begin{matrix}
    D \\ D'
  \end{matrix}
  \  =  \
  \begin{tikzpicture}[anchorbase]
    \pd{0,1};
    \pd{0.5,1};
    \pd{1,1};
    \pd{1.5,1};
    \pd{2,1};
    \pd{2.5,1};
    \pd{3,1};
    \pd{0.5,0};
    \pd{1,0};
    \pd{1.5,0};
    \pd{2,0};
    \pd{2.5,0};
    \draw (1,0) \braidto (0,1);
    \draw (1,0) \braidto (1.5,1);
    \draw (0.5,0) to[out=up,in=up] (2.5,0);
    \draw (1.5,0) \braidto (3,1);
    \draw (2,1) to (2.5,1);
    \pd{0.5,2};
    \pd{1,2};
    \pd{1.5,2};
    \pd{2,2};
    \pd{2.5,2};
    \draw (0,1) \braidto (0.5,2);
    \draw (1,2) to[out=down,in=down] (2,2);
    \draw (0.5,1) to[out=up,in=up] (1.5,1);
    \draw (3,1) \braidto (1.5,2);
    \draw (3,1) \braidto (2.5,2);
  \end{tikzpicture}
  \ ,\quad
  D \star D' =
  \begin{tikzpicture}[anchorbase]
    \pd{0,0};
    \pd{0.5,0};
    \pd{1,0};
    \pd{1.5,0};
    \pd{2,0};
    \pd{0,1};
    \pd{0.5,1};
    \pd{1,1};
    \pd{1.5,1};
    \pd{2,1};
    \draw (0,0) to[out=up,in=up] (2,0);
    \draw (0.5,0) \braidto (0,1);
    \draw (0.5,1) to[out=down,in=down] (1.5,1);
    \draw (1,0) to (1,1);
    \draw (1,0) \braidto (2,1);
  \end{tikzpicture}
  \ ,\quad \text{and} \quad
  D \circ D' = t^2\
  \begin{tikzpicture}[anchorbase]
    \pd{0,0};
    \pd{0.5,0};
    \pd{1,0};
    \pd{1.5,0};
    \pd{2,0};
    \pd{0,1};
    \pd{0.5,1};
    \pd{1,1};
    \pd{1.5,1};
    \pd{2,1};
    \draw (0,0) to[out=up,in=up] (2,0);
    \draw (0.5,0) \braidto (0,1);
    \draw (0.5,1) to[out=down,in=down] (1.5,1);
    \draw (1,0) to (1,1);
    \draw (1,0) \braidto (2,1);
  \end{tikzpicture}
  \ .
\]
The partition category is denoted $\mathrm{Rep}_0(S_t)$ in \cite{Del07} and $\Rep_0(S_t; \kk)$ in \cite{CO11}.

For a linear monoidal category $\cC$, we let $\Kar(\cC)$ denote its additive Karoubi envelope, that is, the idempotent completion of its additive envelope $\Add(\cC)$.  Then $\Kar(\cC)$ is again naturally a linear monoidal category.  \emph{Deligne's category} $\Rep(S_t)$ is the additive Karoubi envelope of $\Par(t)$.  (See \cite[\S8]{Del07} and \cite[\S2.2]{CO11}.)

The following proposition gives a presentation of the partition category.

\begin{prop} \label{Ppresent}
  As a $\kk$-linear monoidal category, the \emph{partition category} $\Par(t)$ is generated by the object $1$ and the morphisms
  \[
    \mu =
    \begin{tikzpicture}[anchorbase]
      \pd{0,0};
      \pd{0.5,0};
      \pd{0.25,0.5};
      \draw (0,0) \braidto (0.25,0.5);
      \draw (0.5,0) \braidto (0.25,0.5);
    \end{tikzpicture}
    \colon 2 \to 1,\quad
    \parcop =
    \begin{tikzpicture}[anchorbase]
      \pd{0,0.5};
      \pd{0.5,0.5};
      \pd{0.25,0};
      \draw (0.25,0) \braidto (0,0.5);
      \draw (0.25,0) \braidto (0.5,0.5);
    \end{tikzpicture}
    \colon 1 \to 2,\quad
    s =
    \begin{tikzpicture}[anchorbase]
      \pd{0,0};
      \pd{0.5,0};
      \pd{0,0.5};
      \pd{0.5,0.5};
      \draw (0,0) \braidto (0.5,0.5);
      \draw (0.5,0) \braidto (0,0.5);
    \end{tikzpicture}
    \colon 2 \to 2,\quad
    \eta =
    \begin{tikzpicture}[anchorbase]
      \pd{0,0.5};
      \draw(0,0.25) to (0,0.5);
    \end{tikzpicture}
    \ \colon 0 \to 1,\quad
    \varepsilon =
    \begin{tikzpicture}[anchorbase]
      \pd{0,0};
      \draw (0,0.25) to (0,0);
    \end{tikzpicture}
    \ \colon 1 \to 0,
  \]
  subject to the following relations:
  \begin{gather} \label{P1}
    \begin{tikzpicture}[anchorbase]
      \pd{0,0};
      \pd{0,0.5};
      \pd{0.25,1};
      \pd{0.5,0.5};
      \draw (0,0) to (0,0.5) \braidto (0.25,1);
      \draw (0.5,0.25) to (0.5,0.5) \braidto (0.25,1);
    \end{tikzpicture}
    \ =\
    \begin{tikzpicture}[anchorbase]
      \pd{0,0};
      \pd{0,0.5};
      \draw (0,0) to (0,0.5);
    \end{tikzpicture}
    \ =\
    \begin{tikzpicture}[anchorbase]
      \pd{0.5,0};
      \pd{0,0.5};
      \pd{0.5,0.5};
      \pd{0.25,1};
      \draw (0.5,0) to (0.5,0.5) \braidto (0.25,1);
      \draw (0,0.25) to (0,0.5) \braidto (0.25,1);
    \end{tikzpicture}
    \ ,\qquad
    \begin{tikzpicture}[anchorbase]
      \pd{0.25,0};
      \pd{0,0.5};
      \pd{0.5,0.5};
      \pd{0,1};
      \draw (0.25,0) \braidto (0,0.5) to (0,1);
      \draw (0.25,0) \braidto (0.5,0.5) to (0.5,0.75);
    \end{tikzpicture}
    \ =\
    \begin{tikzpicture}[anchorbase]
      \pd{0,0};
      \pd{0,0.5};
      \draw (0,0) to (0,0.5);
    \end{tikzpicture}
    \ =\
    \begin{tikzpicture}[anchorbase]
      \pd{0.25,0};
      \pd{0,0.5};
      \pd{0.5,0.5};
      \pd{0.5,1};
      \draw (0.25,0) \braidto (0.5,0.5) to (0.5,1);
      \draw (0.25,0) \braidto (0,0.5) to (0,0.75);
    \end{tikzpicture}
    \ ,\qquad
    \begin{tikzpicture}[anchorbase]
      \pd{0,0};
      \pd{0.5,0};
      \pd{0,0.5};
      \pd{0.5,0.5};
      \pd{1,0.5};
      \pd{0.5,1};
      \pd{1,1};
      \draw (0,0) to (0,0.5) \braidto (0.5,1);
      \draw (0.5,0) to (0.5,1);
      \draw (0.5,0) \braidto (1,0.5) to (1,1);
    \end{tikzpicture}
    \ =\
    \begin{tikzpicture}[anchorbase]
      \pd{0,0};
      \pd{0.5,0};
      \pd{0.25,0.5};
      \pd{0,1};
      \pd{0.5,1};
      \draw (0,0) \braidto (0.25,0.5) \braidto (0,1);
      \draw (0.5,0) \braidto (0.25,0.5) \braidto (0.5,1);
    \end{tikzpicture}
    \ =\
    \begin{tikzpicture}[anchorbase]
      \pd{0.5,0};
      \pd{1,0};
      \pd{0,0.5};
      \pd{0.5,0.5};
      \pd{1,0.5};
      \pd{0,1};
      \pd{0.5,1};
      \draw (0.5,0) \braidto (0,0.5) to (0,1);
      \draw (0.5,0) to (0.5,1);
      \draw (1,0) to (1,0.5) \braidto (0.5,1);
    \end{tikzpicture}
    \ ,
    \\ \label{P2}
    \begin{tikzpicture}[anchorbase]
      \pd{0,0};
      \pd{0.5,0};
      \pd{0,0.5};
      \pd{0.5,0.5};
      \pd{0,1};
      \pd{0.5,1};
      \draw (0,0) \braidto (0.5,0.5) \braidto (0,1);
      \draw (0.5,0) \braidto (0,0.5) \braidto (0.5,1);
    \end{tikzpicture}
    \ =\
    \begin{tikzpicture}[anchorbase]
      \pd{0,0};
      \pd{0.5,0};
      \pd{0,0.5};
      \pd{0.5,0.5};
      \draw (0,0) to (0,0.5);
      \draw (0.5,0) to (0.5,0.5);
    \end{tikzpicture}
    \ ,\qquad
    \begin{tikzpicture}[anchorbase]
      \pd{0,0};
      \pd{0.5,0};
      \pd{1,0};
      \pd{0,0.5};
      \pd{0.5,0.5};
      \pd{1,0.5};
      \pd{0,1};
      \pd{0.5,1};
      \pd{1,1};
      \pd{0,1.5};
      \pd{0.5,1.5};
      \pd{1,1.5};
      \draw (0,0) to (0,0.5) \braidto (0.5,1) \braidto (1,1.5);
      \draw (0.5,0) \braidto (1,0.5) to (1,1) \braidto (0.5,1.5);
      \draw (1,0) \braidto (0.5,0.5) \braidto (0,1) to (0,1.5);
    \end{tikzpicture}
    \ =\
    \begin{tikzpicture}[anchorbase]
      \pd{0,0};
      \pd{0.5,0};
      \pd{1,0};
      \pd{0,0.5};
      \pd{0.5,0.5};
      \pd{1,0.5};
      \pd{0,1};
      \pd{0.5,1};
      \pd{1,1};
      \pd{0,1.5};
      \pd{0.5,1.5};
      \pd{1,1.5};
      \draw (0,0) \braidto (0.5,0.5) \braidto (1,1) to (1,1.5);
      \draw (0.5,0) \braidto (0,0.5) to (0,1) \braidto (0.5,1.5);
      \draw (1,0) to (1,0.5) \braidto (0.5,1) \braidto (0,1.5);
    \end{tikzpicture}
    \ ,
    \\ \label{P3}
    \begin{tikzpicture}[anchorbase]
      \pd{0,0};
      \pd{0,0.5};
      \pd{0.5,0.5};
      \pd{0,1};
      \pd{0.5,1};
      \draw (0,0) to (0,0.5) \braidto (0.5,1);
      \draw (0.5,0.25) to (0.5,0.5) \braidto (0,1);
    \end{tikzpicture}
    \ =\
    \begin{tikzpicture}[anchorbase]
      \pd{0.5,0};
      \pd{0,0.5};
      \pd{0.5,0.5};
      \draw (0,0.25) to (0,0.5);
      \draw (0.5,0) to (0.5,0.5);
    \end{tikzpicture}
    \ ,\qquad
    \begin{tikzpicture}[anchorbase]
      \pd{0,0};
      \pd{0.5,0};
      \pd{1,0};
      \pd{0,0.5};
      \pd{0.5,0.5};
      \pd{1,0.5};
      \pd{0,1};
      \pd{0.5,1};
      \pd{1,1};
      \pd{0,1.5};
      \pd{0.5,1.5};
      \draw (0,0) to (0,0.5) \braidto (0.5,1) to (0.5,1.5);
      \draw (0.5,0) \braidto (1,0.5) to (1,1) \braidto (0.5,1.5);
      \draw (1,0) \braidto (0.5,0.5) \braidto (0,1) to (0,1.5);
    \end{tikzpicture}
    \ =\
    \begin{tikzpicture}[anchorbase]
      \pd{0,0};
      \pd{0.5,0};
      \pd{1,0};
      \pd{0.5,0.5};
      \pd{1,0.5};
      \pd{0.5,1};
      \pd{1,1};
      \draw (0,0) \braidto (0.5,0.5) \braidto (1,1);
      \draw (0.5,0) to (0.5,0.5);
      \draw (1,0) to (1,0.5) \braidto (0.5,1);
    \end{tikzpicture}
    \ ,\qquad
    \begin{tikzpicture}[anchorbase]
      \pd{0,0};
      \pd{0.5,0};
      \pd{0,0.5};
      \pd{0.5,0.5};
      \pd{0,1};
      \draw (0,0) \braidto (0.5,0.5) to (0.5,0.75);
      \draw (0.5,0) \braidto (0,0.5) to (0,1);
    \end{tikzpicture}
    \ =\
    \begin{tikzpicture}[anchorbase]
      \pd{0,0};
      \pd{0.5,0};
      \pd{0.5,0.5};
      \draw (0,0) to (0,0.25);
      \draw (0.5,0) to (0.5,0.5);
    \end{tikzpicture}
    \ ,\qquad
    \begin{tikzpicture}[anchorbase]
      \pd{0,0};
      \pd{0.5,0};
      \pd{0,0.5};
      \pd{0.5,0.5};
      \pd{1,0.5};
      \pd{0,1};
      \pd{0.5,1};
      \pd{1,1};
      \pd{0,1.5};
      \pd{0.5,1.5};
      \pd{1,1.5};
      \draw (0,0) to (0,0.5) \braidto (0.5,1) \braidto (1,1.5);
      \draw (0.5,0) to (0.5,0.5) \braidto (0,1) to (0,1.5);
      \draw (0.5,0) \braidto (1,0.5) to (1,1) \braidto (0.5,1.5);
    \end{tikzpicture}
    \ =\
    \begin{tikzpicture}[anchorbase]
      \pd{0.5,0};
      \pd{1,0};
      \pd{0.5,0.5};
      \pd{1,0.5};
      \pd{0,1};
      \pd{0.5,1};
      \pd{1,1};
      \draw (0.5,0) \braidto (1,0.5) to (1,1);
      \draw (1,0) \braidto (0.5,0.5) \braidto (0,1);
      \draw (0.5,0.5) to (0.5,1);
    \end{tikzpicture}
    \ ,
    \\ \label{P4}
    \begin{tikzpicture}[anchorbase]
      \pd{0,0};
      \pd{0.5,0};
      \pd{0,0.5};
      \pd{0.5,0.5};
      \pd{0.25,1};
      \draw (0,0) \braidto (0.5,0.5) \braidto (0.25,1);
      \draw (0.5,0) \braidto (0,0.5) \braidto (0.25,1);
    \end{tikzpicture}
    \ =\
    \begin{tikzpicture}[anchorbase]
      \pd{0,0};
      \pd{0.5,0};
      \pd{0.25,0.5};
      \draw (0,0) \braidto (0.25,0.5);
      \draw (0.5,0) \braidto (0.25,0.5);
    \end{tikzpicture}
    \ ,\qquad
    \begin{tikzpicture}[anchorbase]
      \pd{0.25,0};
      \pd{0,0.5};
      \pd{0.5,0.5};
      \pd{0.25,1};
      \draw (0.25,0) \braidto (0,0.5) \braidto (0.25,1);
      \draw (0.25,0) \braidto (0.5,0.5) \braidto (0.25,1);
    \end{tikzpicture}
    \ = \
    \begin{tikzpicture}[anchorbase]
      \pd{0,0};
      \pd{0,0.5};
      \draw (0,0) to (0,0.5);
    \end{tikzpicture}
    \ ,\qquad
    \begin{tikzpicture}[anchorbase]
      \pd{0,0};
      \draw (0,-0.25) to (0,0.25);
    \end{tikzpicture}
    = t 1_0.
  \end{gather}
\end{prop}

In fact, one only needs one of the equalities in the first string of equalities in \cref{P1}.  The other then follows using the first relation in \cref{P4} and the first relation in \cref{P3}.  The reader who prefers a more traditional algebraic formulation of the above presentation of $\Par(t)$ can find this in \cite[Th.~2.1]{Com16}.

\begin{proof}
  This result is proved in \cite[Th.~2.1]{Com16}.  While it is assumed throughout \cite{Com16} that $\kk$ is a field of characteristic not equal to $2$, these restrictions are not needed in the proof of \cite[Th.~2.1]{Com16}.  The essence of the proof is noting that $\Par(t)$ is isomorphic to the category obtained from the $\kk$-linearization of a skeleton of the category $2\mathpzc{Cob}$ of 2-dimensional cobordisms by factoring out by the second and third relations in \cref{P4}.  Then the result is deduced from the presentation of $2\mathpzc{Cob}$ described in \cite[\S 1.4]{Koc04}.
\end{proof}

The relations \cref{P1} are equivalent to the statement that $(1,\mu,\eta,\parcop,\varepsilon)$ is a Frobenius object (see, for example, \cite[Prop.~2.3.24]{Koc04}).  Relations \cref{P2,P3} are precisely the statement that $s$ equips $\Par(t)$ with the structure of a symmetric monoidal category (see, for example, \cite[\S1.3.27, \S1.4.35]{Koc04}).  Then the relations \cref{P4} are precisely the statements that the Frobenius object $1$ is commutative, special, and of dimension $t$, respectively.  Thus, \cref{Ppresent} states that $\Par(t)$ is the free $\kk$-linear symmetric monoidal category generated by a $t$-dimensional special commutative Frobenius object.

The endomorphism algebra $P_k(t) := \End_{\Par(t)}(k)$ is called the \emph{partition algebra}.  We have a natural algebra homomorphism
\begin{equation} \label{garage}
  \kk S_k \to P_k(t),
\end{equation}
mapping $\tau \in S_k$ to the partition with blocks $\{i,\tau(i)'\}$, $1 \le i \le k$.

Let $V = \kk^n$ be the permutation representation of $S_n$ and let $\mathbf{1}_n$ denote the one-dimensional trivial $S_n$-module.  As explained in \cite[\S2.4]{Com16}, there is a strong monoidal functor
\begin{equation} \label{Phidef}
  \Phi_n \colon \Par(n) \to S_n\md
\end{equation}
defined on generators by setting $\Phi_n(1) = V$ and
\begin{align*}
  \Phi_n(\mu) &\colon V \otimes V \to V,& v_i\otimes v_j &\mapsto \delta_{i,j}v_i, \\
  \Phi_n(\eta) &\colon \mathbf{1}_n \to V,& 1 &\mapsto \textstyle \sum_{i=1}^{n} v_i, \\
  \Phi_n(\parcop) &\colon V \to V \otimes V,& v_i &\mapsto v_i \otimes v_i, \\
  \Phi_n(\varepsilon) &\colon V \to \mathbf{1}_n,& v_i &\mapsto 1, \\
  \Phi_n(s) &\colon V \otimes V \to V \otimes V,& v_i\otimes v_j &\mapsto v_j\otimes v_i.
\end{align*}
The proposition below is a generalization of the duality property of the partition algebra mentioned in the introduction.

\begin{prop} \label{bee}
  \begin{enumerate}
    \item The functor $\Phi_n$ is full.
    \item The induced map
      \[
        \Hom_{\Par(n)}(k,\ell)\to\Hom_{S_n}(V^{\otimes k}, V^{\otimes \ell})
      \]
      is an isomorphism if and only if $k + \ell \leq n$.
  \end{enumerate}
\end{prop}

\begin{proof}
  This is proved in \cite[Th.~2.3]{Com16}.  While it is assumed throughout \cite{Com16} that $\kk$ is a field of characteristic not equal to $2$, that assumption is not needed in the proof of \cite[Th.~2.3]{Com16}.  When $k=\ell$, the current proposition reduces to a statement about the partition algebra; see \cite[Th.~3.6]{HR05}.
\end{proof}

\section{The Heisenberg category\label{sec:Heis}}

In this section we define the Heisenberg category originally introduced by Khovanov in \cite{Kho14}.  This is the central charge $-1$ case of a more general Heisenberg category described in \cite{MS18,Bru18}.  We give here the efficient presentation of this category described in \cite[Rem.~1.5(2)]{Bru18}.

The \emph{Heisenberg category} $\Heis$ is the strict $\kk$-linear monoidal category generated by two objects $\uparrow$, $\downarrow$, (we use horizontal juxtaposition to denote the tensor product) and morphisms
\[
  \begin{tikzpicture}[anchorbase]
    \draw[->] (0.6,0) -- (0,0.6);
    \draw[->] (0,0) -- (0.6,0.6);
  \end{tikzpicture}
  \colon \uparrow \uparrow\ \to\ \uparrow \uparrow
  , \quad
  \begin{tikzpicture}[anchorbase]
    \draw[->] (0,.2) -- (0,0) arc (180:360:.3) -- (.6,.2);
  \end{tikzpicture}
  \ \colon \one \to\ \downarrow \uparrow
  , \quad
  \begin{tikzpicture}[anchorbase]
    \draw[->] (0,-.2) -- (0,0) arc (180:0:.3) -- (.6,-.2);
  \end{tikzpicture}
  \ \colon \uparrow \downarrow\ \to \one
  , \quad
  \begin{tikzpicture}[anchorbase]
    \draw[<-] (0,.2) -- (0,0) arc (180:360:.3) -- (.6,.2);
  \end{tikzpicture}
  \ \colon \one \to\ \uparrow \downarrow
  , \quad
  \begin{tikzpicture}[anchorbase]
    \draw[<-] (0,0) -- (0,.2) arc (180:0:.3) -- (.6,0);
  \end{tikzpicture}
  \ \colon \downarrow \uparrow\ \to \one,
\]
where $\one$ denotes the unit object, subject to the relations
\begin{gather} \label{H1}
  \begin{tikzpicture}[anchorbase]
    \draw[->] (0.3,0) \braidto (-0.3,0.6) \braidto (0.3,1.2);
    \draw[->] (-0.3,0) to[out=up,in=down] (0.3,0.6) \braidto (-0.3,1.2);
  \end{tikzpicture}
  \ =\
  \begin{tikzpicture}[anchorbase]
    \draw[->] (-0.2,0) -- (-0.2,1.2);
    \draw[->] (0.2,0) -- (0.2,1.2);
  \end{tikzpicture}
  \ ,\qquad
  \begin{tikzpicture}[anchorbase]
    \draw[->] (0.4,0) -- (-0.4,1.2);
    \draw[->] (0,0) \braidto (-0.4,0.6) \braidto (0,1.2);
    \draw[->] (-0.4,0) -- (0.4,1.2);
  \end{tikzpicture}
  \ =\
  \begin{tikzpicture}[anchorbase]
    \draw[->] (0.4,0) -- (-0.4,1.2);
    \draw[->] (0,0) \braidto (0.4,0.6) \braidto (0,1.2);
    \draw[->] (-0.4,0) -- (0.4,1.2);
  \end{tikzpicture}
  \ ,
  \\ \label{H2}
  \begin{tikzpicture}[anchorbase]
    \draw[->] (0,0) -- (0,0.6) arc(180:0:0.2) -- (0.4,0.4) arc(180:360:0.2) -- (0.8,1);
  \end{tikzpicture}
  \ =\
  \begin{tikzpicture}[anchorbase]
    \draw[->] (0,0) -- (0,1);
  \end{tikzpicture}
  \ ,\qquad
  \begin{tikzpicture}[anchorbase]
    \draw[->] (0,1) -- (0,0.4) arc(180:360:0.2) -- (0.4,0.6) arc(180:0:0.2) -- (0.8,0);
  \end{tikzpicture}
  \ =\
  \begin{tikzpicture}[anchorbase]
    \draw[<-] (0,0) -- (0,1);
  \end{tikzpicture}
  \ ,
  \\ \label{H3}
  \begin{tikzpicture}[anchorbase]
    \draw[->] (-0.3,-0.6) \braidto (0.3,0) \braidto (-0.3,0.6);
    \draw[<-] (0.3,-0.6) \braidto (-0.3,0) \braidto (0.3,0.6);
  \end{tikzpicture}
  \ =\
  \begin{tikzpicture}[anchorbase]
    \draw[->] (-0.2,-0.6) to (-0.2,0.6);
    \draw[<-] (0.2,-0.6) to (0.2,0.6);
  \end{tikzpicture}
  \ ,\qquad
  \begin{tikzpicture}[anchorbase]
    \draw[<-] (-0.3,-0.6) \braidto (0.3,0) \braidto (-0.3,0.6);
    \draw[->] (0.3,-0.6) \braidto (-0.3,0) \braidto (0.3,0.6);
  \end{tikzpicture}
  \ =\
  \begin{tikzpicture}[anchorbase]
    \draw[<-] (-0.2,-0.6) to (-0.2,0.6);
    \draw[->] (0.2,-0.6) to (0.2,0.6);
  \end{tikzpicture}
  -
  \begin{tikzpicture}[anchorbase]
    \draw[<-] (-0.3,0) to (-0.3,0.2) arc(180:0:0.3) to (0.3,0);
    \draw[->] (-0.3,1.2) to (-0.3,1) arc(-180:0:0.3) to (0.3,1.2);
  \end{tikzpicture}
  \ ,\qquad
  \begin{tikzpicture}[anchorbase]
  	\draw[<-] (0,0.6) to (0,0.3);
  	\draw (-0.3,-0.2) to [out=180,in=-90](-.5,0);
  	\draw (-0.5,0) to [out=90,in=180](-.3,0.2);
  	\draw (-0.3,.2) to [out=0,in=90](0,-0.3);
  	\draw (0,-0.3) to (0,-0.6);
  	\draw (0,0.3) to [out=-90,in=0] (-.3,-0.2);
  \end{tikzpicture}
  \ = 0,
  \qquad
  \begin{tikzpicture}[anchorbase]
    \draw[->] (0,0.3) arc(90:450:0.3);
  \end{tikzpicture}
  = 1_\one.
\end{gather}
Here the left and right crossings are defined by
\[
  \begin{tikzpicture}[anchorbase]
    \draw[<-] (0,0) -- (0.6,0.6);
    \draw[->] (0.6,0) -- (0,0.6);
  \end{tikzpicture}
  \ :=\
  \begin{tikzpicture}[anchorbase]
    \draw[->] (-0.2,-0.3) to (0.2,0.3);
    \draw[<-] (-0.6,-0.3) to[out=up,in=135,looseness=2] (0,0) to[out=-45,in=down,looseness=2] (0.6,0.3);
  \end{tikzpicture}
  \ ,\qquad
  \begin{tikzpicture}[anchorbase]
    \draw[->] (0,0) -- (0.6,0.6);
    \draw[<-] (0.6,0) -- (0,0.6);
  \end{tikzpicture}
  \ :=\
  \begin{tikzpicture}[anchorbase]
    \draw[->] (0.2,-0.3) to (-0.2,0.3);
    \draw[<-] (0.6,-0.3) to[out=up,in=45,looseness=2] (0,0) to[out=225,in=down,looseness=2] (-0.6,0.3);
  \end{tikzpicture}
  \ .
\]

The category $\Heis$ is strictly pivotal, meaning that morphisms are invariant under isotopy (see \cite[Th.~1.3(ii),(iii)]{Bru18}).  The relations \cref{H3} imply that
\begin{equation} \label{key}
  \downarrow \uparrow\ \cong\ \uparrow \downarrow \oplus \one.
\end{equation}
In addition, we have the following \emph{bubble slide} relations (see \cite[p.~175]{Kho14}, \cite[(13), (19)]{Bru18}):
\begin{equation} \label{bubslide}
  \begin{tikzpicture}[anchorbase]
    \draw[->] (0,0.3) arc(450:90:0.3);
    \draw[->] (0.6,-0.5) to (0.6,0.5);
  \end{tikzpicture}
  \ =\
  \begin{tikzpicture}[anchorbase]
    \draw[->] (0,0.3) arc(450:90:0.3);
    \draw[->] (-0.6,-0.5) to (-0.6,0.5);
  \end{tikzpicture}
  \ +\
  \begin{tikzpicture}[anchorbase]
    \draw[->] (0,-0.5) to (0,0.5);
  \end{tikzpicture}
  \qquad \text{and} \qquad
  \begin{tikzpicture}[anchorbase]
    \draw[->] (0,0.3) arc(450:90:0.3);
    \draw[<-] (0.6,-0.5) to (0.6,0.5);
  \end{tikzpicture}
  \ =\
  \begin{tikzpicture}[anchorbase]
    \draw[->] (0,0.3) arc(450:90:0.3);
    \draw[<-] (-0.6,-0.5) to (-0.6,0.5);
  \end{tikzpicture}
  \ -\
  \begin{tikzpicture}[anchorbase]
    \draw[<-] (0,-0.5) to (0,0.5);
  \end{tikzpicture}.
\end{equation}
We can define downward crossings
\[
  \begin{tikzpicture}[anchorbase]
    \draw[<-] (0,0) -- (0.6,0.6);
    \draw[<-] (0.6,0) -- (0,0.6);
  \end{tikzpicture}
  \ :=\
  \begin{tikzpicture}[anchorbase]
    \draw[<-] (-0.2,-0.3) to (0.2,0.3);
    \draw[<-] (-0.6,-0.3) to[out=up,in=135,looseness=2] (0,0) to[out=-45,in=down,looseness=2] (0.6,0.3);
  \end{tikzpicture}
\]
and then we have
\begin{equation} \label{genbraid}
  \begin{tikzpicture}[anchorbase]
    \draw (0.4,0) -- (-0.4,1.2);
    \draw (0,0) \braidto (-0.4,0.6) \braidto (0,1.2);
    \draw (-0.4,0) -- (0.4,1.2);
  \end{tikzpicture}
  \ =\
  \begin{tikzpicture}[anchorbase]
    \draw (0.4,0) -- (-0.4,1.2);
    \draw (0,0) \braidto (0.4,0.6) \braidto (0,1.2);
    \draw (-0.4,0) -- (0.4,1.2);
  \end{tikzpicture}
  \quad \text{for all possible orientations of the strands}
\end{equation}
(see \cite[p.~175]{Kho14}, \cite[(20)]{Bru18}).

For $1 \le i \le k-1$, let $s_i \in S_k$ denote the simple transposition of $i$ and $i+1$.  We have natural algebra homomorphisms
\begin{equation} \label{SunnyD}
  \kk S_k \to \End_\Heis(\uparrow^k)
  \quad \text{and} \quad
  \kk S_k \to \End_\Heis(\downarrow^k),
\end{equation}
where $s_i$ is mapped to the crossing of strands $i$ and $i+1$, numbering strands \emph{from right to left}.

Let $\Peis$ denote the full $\kk$-linear monoidal subcategory of $\Heis$ generated by $\uparrow \downarrow$.  It follows immediately from \cref{bubslide} that
\[
  \begin{tikzpicture}[anchorbase]
    \draw[->] (0,0.3) arc(450:90:0.3);
    \draw[->] (0.6,-0.5) to (0.6,0.5);
    \draw[<-] (0.9,-0.5) to (0.9,0.5);
  \end{tikzpicture}
  \ =\
  \begin{tikzpicture}[anchorbase]
    \draw[->] (0,0.3) arc(450:90:0.3);
    \draw[<-] (-0.6,-0.5) to (-0.6,0.5);
    \draw[->] (-0.9,-0.5) to (-0.9,0.5);
  \end{tikzpicture}
  \ .
\]
In other words, the clockwise bubble is strictly central in $\Peis$.  Thus, fixing $t \in \kk$, we can define $\Peis(t)$ to be the quotient of $\Peis$ by the additional relation
\begin{equation} \label{H4}
  \begin{tikzpicture}[anchorbase]
    \draw[->] (0,0.3) arc(450:90:0.3);
  \end{tikzpicture}
  \ = t 1_\one.
\end{equation}

For additive categories $\cC_i$, $i \in I$, the direct product category $\prod_{i \in I} \cC_i$ has objects $(X_i)_{i \in I}$, where $X_i \in \cC_i$.  Morphisms $(X_i)_{i \in I} \to (Y_i)_{i \in I}$ are $(f_i)_{i \in I}$, where $f_i \in \Hom_{\cC_i}(X_i,Y_i)$, with componentwise composition.  The direct sum category $\bigoplus_{i \in I} \cC_i$ is the full subcategory of $\prod_{i \in I} \cC_i$ on objects $(X_i)_{i \in I}$ where all but finitely many of the $X_i$ are the zero object.

We now recall the action of $\Heis$ on the category of $S_n$-modules first defined by Khovanov \cite[\S3.3]{Kho14}.  We begin by defining a strong $\kk$-linear monoidal functor
\[
  \Theta \colon \Heis \to \prod_{m \in \N} \left( \bigoplus_{n \in \N} (S_n,S_m)\bmd \right).
\]
The tensor product structure on the codomain is given by the usual tensor product of bimodules, where we define the tensor product $M \otimes N$ of $M \in (S_n, S_m)\bmd$ and $N \in (S_k,S_\ell)\bmd$ to be zero when $m \ne k$.  We adopt the convention that $S_0$ is the trivial group, so that $S_0\md$ is the category of $\kk$-vector spaces.  For $0 \le m,k \le n$, let ${}_k(n)_m$ denote $\kk S_n$, considered as an $(S_k,S_m)$-bimodule.  We will omit the subscript $k$ or $m$ when $k=n$ or $m=n$, respectively.  We denote tensor product of such bimodules by juxtaposition.  For instance $(n)_{n-1}(n)$ denotes $\kk S_n \otimes_{n-1} \kk S_n$, considered as an $(S_n,S_n)$-bimodule, where we write $\otimes_m$ for the tensor product over $\kk S_m$.  We adopt the convention that $s_i s_{i+1} \dotsm s_j = 1$ when $i > j$.  Then the elements
\begin{equation} \label{gi-def}
  g_i = s_i s_{i+1} \dotsm s_{n-1},\quad i=1,\dotsc,n,
\end{equation}
form a complete set of left coset representatives of $S_{n-1}$ in $S_n$.

On objects, we define
\[
  \Theta(\uparrow) = \left( (n)_{n-1} \right)_{n \ge 1},\qquad
  \Theta(\downarrow) = \left( {}_{n-1}(n) \right)_{n \ge 1}.
\]
On the generating morphisms, we define
\begin{align*}
  \Theta
  \left(
    \begin{tikzpicture}[anchorbase]
      \draw[->] (0.6,0) -- (0,0.6);
      \draw[->] (0,0) -- (0.6,0.6);
    \end{tikzpicture}
  \right)
  &=
  \Big(
    (n)_{n-2} \to (n)_{n-2},\ g \mapsto g s_{n-1}
  \Big)_{n \ge 2},
  \\
  \Theta
  \left(
    \begin{tikzpicture}[anchorbase]
      \draw[->] (0,.2) -- (0,0) arc (180:360:.3) -- (.6,.2);
    \end{tikzpicture}
  \right)
  &=
  \Big(
    (n-1) \to {}_{n-1}(n)_{n-1},\ g \mapsto g
  \Big)_{n \ge 1},
  \\
  \Theta
  \left(
    \begin{tikzpicture}[anchorbase]
      \draw[->] (0,-.2) -- (0,0) arc (180:0:.3) -- (.6,-.2);
    \end{tikzpicture}
  \right)
  &=
  \Big(
    (n)_{n-1}(n) \to (n),\ g \otimes h \mapsto gh
  \Big)_{n \ge 1},
  \\
  \Theta
  \left(
    \begin{tikzpicture}[anchorbase]
      \draw[<-] (0,.2) -- (0,0) arc (180:360:.3) -- (.6,.2);
    \end{tikzpicture}
  \right)
  &=
  \Big( \textstyle
    (n) \to (n)_{n-1}(n),\ g \mapsto \sum_{i=1}^n g_i \otimes g_i^{-1}g = \sum_{i=1}^n g g_i \otimes g_i^{-1}
  \Big)_{n \ge 1},
  \\
  \Theta
  \left(
    \begin{tikzpicture}[anchorbase]
      \draw[<-] (0,0) -- (0,.2) arc (180:0:.3) -- (.6,0);
    \end{tikzpicture}
  \right)
  &=
  \left(
    {}_{n-1}(n)_{n-1} \to (n-1),\ g \mapsto
    \begin{cases}
      g & \text{if } g \in S_{n-1}, \\
      0 & \text{if } g \in S_n \setminus S_{n-1}
    \end{cases}
  \right)_{n \ge 1}.
\end{align*}
One can then compute that
\begin{align*}
  \Theta
  \left(
    \begin{tikzpicture}[anchorbase]
      \draw[<-] (0,0) -- (0.6,0.6);
      \draw[->] (0.6,0) -- (0,0.6);
    \end{tikzpicture}
  \right)
  &=
  \left(
    {}_{n-1}(n)_{n-1} \to (n-1)_{n-2}(n-1),\
    \begin{cases}
      g s_{n-1} h \mapsto g \otimes h, & g,h \in S_{n-1}, \\
      g \mapsto 0, & g \in S_{n-1}
    \end{cases}
  \right)_{n \ge 2},
  \\
  \Theta
  \left(
    \begin{tikzpicture}[anchorbase]
      \draw[->] (0,0) -- (0.6,0.6);
      \draw[<-] (0.6,0) -- (0,0.6);
    \end{tikzpicture}
  \right)
  &=
  \Big(
    (n-1)_{n-2}(n-1) \to {}_{n-1}(n)_{n-1},\ g \otimes h \mapsto g s_{n-1} h
  \Big)_{n \ge 2},\\
  \Theta
  \left(
    \begin{tikzpicture}[anchorbase]
      \draw[<-] (0,0) -- (0.6,0.6);
      \draw[<-] (0.6,0) -- (0,0.6);
    \end{tikzpicture}
  \right)
  &=
  \Big(
    {}_{n-2}(n) \to {}_{n-2}(n),\ g \mapsto s_{n-1} g
  \Big)_{n \ge 2}.
\end{align*}

Restricting to $\Peis$ yields a functor, which we denote by the same symbol,
\[
  \Theta \colon \Peis \to \bigoplus_{m \in \N} (S_m,S_m)\bmd.
\]
Recall that $\mathbf{1}_n$ denotes the one-dimensional trivial $S_n$-module.  Then the functor $- \otimes_n \mathbf{1}_n$ of tensoring on the right with $\mathbf{1}_n$ gives a functor
\[
  \bigoplus_{m \in \N} (S_m,S_m)\bmd
  \xrightarrow{- \otimes_n \mathbf{1}_n} S_n\md.
\]
Here we define $M \otimes_n \mathbf{1}_n = 0$ for $M \in (S_m,S_m)\bmd$, $m \ne n$.  Consider the composition
\[
  \Peis
  \xrightarrow{\Theta} \bigoplus_{m \in \N} (S_m,S_m)\bmd
  \xrightarrow{- \otimes_n \mathbf{1}_n} S_n\md.
\]
It is straightforward to verify that the image of the relation \cref{H4} under this composition holds in $S_n\md$ with $t=n$.  Therefore, the composition factors through $\Peis(n)$ to give us our action functor:
\begin{equation} \label{Omegadef}
  \Omega_n \colon \Peis(n) \to S_n\md.
\end{equation}
Note that the functor $\Omega_n$ is not monoidal, since the functor $- \otimes_n \mathbf{1}_n$ is not.

\section{Existence of the embedding functor\label{sec:functor}}

In this section we define a functor from the partition category to the Heisenberg category.  We will later show, in \cref{faithful,appendix}, that this functor is faithful.  As we will see in \cref{sec:faithful}, the existence of this functor arises from the fact that the composition $\Ind_{n-1}^n \circ \Res_{n-1}^n$ of the induction functor $\Ind_{n-1}^n \colon S_{n-1}\md \to S_n\md$ and the restriction functor $\Res_{n-1}^n \colon S_n\md \to S_{n-1}\md$ is naturally isomorphic to the functor of tensoring with the permutation module of $S_n$.

\begin{theo} \label{functordef}
  There is a strict linear monoidal functor $\Psi_t \colon \Par(t) \to \Peis(t)$ defined on objects by $k \mapsto (\uparrow \downarrow)^k$ and on generating morphisms by
  \begin{gather*}
    \mu =
    \begin{tikzpicture}[anchorbase]
      \pd{0,0};
      \pd{0.5,0};
      \pd{0.25,0.5};
      \draw (0,0) \braidto (0.25,0.5);
      \draw (0.5,0) \braidto (0.25,0.5);
    \end{tikzpicture}
    \mapsto
    \begin{tikzpicture}[anchorbase]
      \draw[->] (0,0) \braidto (0.33,1);
      \draw[<-] (1,0) \braidto (0.67,1);
      \draw[<-] (0.33,0) to (0.33,0.1) to[out=up,in=up,looseness=2] (0.67,0.1) to (0.67,0);
    \end{tikzpicture}
    \ ,\qquad
    \parcop =
    \begin{tikzpicture}[anchorbase]
      \pd{0,0.5};
      \pd{0.5,0.5};
      \pd{0.25,0};
      \draw (0.25,0) \braidto (0,0.5);
      \draw (0.25,0) \braidto (0.5,0.5);
    \end{tikzpicture}
    \mapsto
    \begin{tikzpicture}[anchorbase]
      \draw[->] (0.33,0) \braidto (0,1);
      \draw[<-] (0.67,0) \braidto (1,1);
      \draw[->] (0.33,1) to (0.33,0.9) to[out=down,in=down,looseness=2] (0.67,0.9) to (0.67,1);
    \end{tikzpicture}
    \ ,\qquad
    s =
    \begin{tikzpicture}[anchorbase]
      \pd{0,0};
      \pd{0.5,0};
      \pd{0,0.5};
      \pd{0.5,0.5};
      \draw (0,0) \braidto (0.5,0.5);
      \draw (0.5,0) \braidto (0,0.5);
    \end{tikzpicture}
    \mapsto
    \begin{tikzpicture}[anchorbase]
      \draw[->] (0,0) \braidto (1,1);
      \draw[<-] (0.5,0) \braidto (1.5,1);
      \draw[->] (1,0) \braidto (0,1);
      \draw[<-] (1.5,0) \braidto (0.5,1);
    \end{tikzpicture}
    \ +\
    \begin{tikzpicture}[anchorbase]
      \draw[->] (1.2,0) -- (1.2,1);
      \draw[->] (1.5,1) -- (1.5,0.9) arc (180:360:.25) -- (2,1);
      \draw[<-] (1.5,0) -- (1.5,0.1) arc (180:0:.25) -- (2,0);
      \draw[<-] (2.3,0) -- (2.3,1);
    \end{tikzpicture}
    \ ,
    \\
    \eta =
    \begin{tikzpicture}[anchorbase]
      \pd{0,0.5};
      \draw (0,0.25) to (0,0.5);
    \end{tikzpicture}
    \mapsto
    \begin{tikzpicture}[anchorbase]
      \draw[<-] (0,1) -- (0,0.9) arc (180:360:.25) -- (0.5,1);
    \end{tikzpicture}
    \ ,\qquad
    \varepsilon =
    \begin{tikzpicture}[anchorbase]
      \pd{0,0};
      \draw (0,0) to (0,0.25);
    \end{tikzpicture}
    \mapsto
    \begin{tikzpicture}[anchorbase]
      \draw[->] (0,0) -- (0,0.1) arc (180:0:.25) -- (0.5,0);
    \end{tikzpicture}
    \ .
  \end{gather*}
\end{theo}

\begin{proof}
  It suffices to prove that the functor $\Psi_t$ preserves the relations \cref{P1,P2,P3,P4}.  Since the objects $\uparrow$ and $\downarrow$ are both left and right dual to each other, the fact that $\Psi_t$ preserves the relations \cref{P1} corresponds to the well-known fact that when $X$ and $Y$ are objects in a monoidal category that are both left and right dual to each other, then $XY$ is a Frobenius object.  Alternatively, one easily can verify directly that $\Psi_t$ preserves the relations \cref{P1}.  This uses only the isotopy invariance in $\Heis$ (i.e.\ the fact that $\Heis$ is strictly pivotal).

  To verify the first relation in \cref{P2}, we compute the image of the left-hand side.  Since left curls in $\Peis(t)$ are zero by \cref{H3}, this image is
  \[
    \Psi_t(s) \circ \Psi_t(s)
    =
    \begin{tikzpicture}[anchorbase]
      \draw[->] (0,0) \braidto (1,1) \braidto (0,2);
      \draw[<-] (0.5,0) \braidto (1.5,1) \braidto (0.5,2);
      \draw[->] (1,0) \braidto (0,1) \braidto (1,2);
      \draw[<-] (1.5,0) \braidto (0.5,1) \braidto (1.5,2);
    \end{tikzpicture}
    \ +\
    \begin{tikzpicture}[anchorbase]
      \draw[->] (-0.8,-1) to (-0.8,1);
      \draw[<-] (0.8,-1) to (0.8,1);
      \draw[<-] (-0.3,-1) to (-0.3,-0.9) arc(180:0:0.3) to (0.3,-1);
      \draw[->] (-0.3,1) to (-0.3,0.9) arc(180:360:0.3) to (0.3,1);
      \draw[->] (0.3,0) arc(0:360:0.3);
    \end{tikzpicture}
    \ \underset{\cref{H3}}{\overset{\cref{H1}}{=}}\
    \begin{tikzpicture}[anchorbase]
      \draw[->] (0,0) to (0,2);
      \draw[<-] (0.5,0) \braidto (1,1) \braidto (0.5,2);
      \draw[->] (1,0) \braidto (0.5,1) \braidto (1,2);
      \draw[<-] (1.5,0) to (1.5,2);
    \end{tikzpicture}
    \ +\
    \begin{tikzpicture}[anchorbase]
      \draw[->] (-0.8,-1) to (-0.8,1);
      \draw[<-] (0.8,-1) to (0.8,1);
      \draw[<-] (-0.3,-1) to (-0.3,-0.9) arc(180:0:0.3) to (0.3,-1);
      \draw[->] (-0.3,1) to (-0.3,0.9) arc(180:360:0.3) to (0.3,1);
      \draw[->] (0.3,0) arc(0:360:0.3);
    \end{tikzpicture}
    \ \overset{\cref{H3}}{=}
    1_{(\uparrow \downarrow)^2}.
  \]

  Next we verify the second relation in \cref{P2}.  First we compute
  \begin{multline*}
    \Psi_t(1_1 \otimes s) \circ \Psi_t(s \otimes 1_1)
    \\
    =
    \begin{tikzpicture}[anchorbase]
      \draw[->] (0,0) to[out=up,in=240] (2,1);
      \draw[<-] (0.5,0) to[out=60,in=down] (2.5,1);
      \draw[->] (1,0) \braidto (0,1);
      \draw[<-] (1.5,0) \braidto (0.5,1);
      \draw[->] (2,0) \braidto (1,1);
      \draw[<-] (2.5,0) \braidto (1.5,1);
    \end{tikzpicture}
    \ +\
    \begin{tikzpicture}[anchorbase]
      \draw[->] (0,0) to (0,1);
      \draw[<-] (0.5,0) to (0.5,0.1) to[out=up,in=up,looseness=1.5] (1,0.1) to (1,0);
      \draw[<-] (1.5,0) \braidto (2.5,1);
      \draw[->] (2,0) \braidto (1,1);
      \draw[<-] (2.5,0) \braidto (1.5,1);
      \draw[->] (0.5,1) to[out=down,in=down] (2,1);
    \end{tikzpicture}
    \ +\
    \begin{tikzpicture}[anchorbase]
      \draw[->] (0,0) \braidto (1,1);
      \draw[<-] (0.5,0) to[out=up,in=up] (2,0);
      \draw[->] (1,0) \braidto (0,1);
      \draw[<-] (1.5,0) \braidto (0.5,1);
      \draw[<-] (2.5,0) to (2.5,1);
      \draw[->] (1.5,1) to (1.5,0.9) to[out=down,in=down,looseness=1.5] (2,0.9) to (2,1);
    \end{tikzpicture}
    \ +\
    \begin{tikzpicture}[anchorbase]
      \draw[->] (0,0) to (0,1);
      \draw[<-] (0.5,0) to (0.5,0.1) to[out=up,in=up,looseness=1.5] (1,0.1) to (1,0);
      \draw[<-] (1.5,0) to (1.5,0.1) to[out=up,in=up,looseness=1.5] (2,0.1) to (2,0);
      \draw[<-] (2.5,0) to (2.5,1);
      \draw[->] (0.5,1) to (0.5,0.9) to[out=down,in=down,looseness=1.5] (1,0.9) to (1,1);
      \draw[->] (1.5,1) to (1.5,0.9) to[out=down,in=down,looseness=1.5] (2,0.9) to (2,1);
    \end{tikzpicture}
    \ .
  \end{multline*}
  Thus, using the fact that left curls are zero and counterclockwise bubbles are $1_\one$ by \cref{H3}, we have
  \begin{align*}
    &\Psi_t(s \otimes 1_1)\circ \Psi_t(1_1 \otimes s) \circ \Psi_t(s \otimes 1_1) \\
    &=
    \begin{tikzpicture}[anchorbase]
      \draw[->] (0,-0.75) to[out=up,in=240] (2,0.75);
      \draw[<-] (0.5,-0.75) to[out=60,in=down] (2.5,0.75);
      \draw[->] (1,-0.75) \braidto (2.1,0) \braidto (1,0.75);
      \draw[<-] (1.5,-0.75) to[out=60,in=down] (2.5,0) to[out=up,in=-60] (1.5,0.75);
      \draw[->] (2,-0.75) to[out=120,in=down] (0,0.75);
      \draw[<-] (2.5,-0.75) to[out=up,in=-60] (0.5,0.75);
    \end{tikzpicture}
    +
    \begin{tikzpicture}[anchorbase]
      \draw[->] (0,-0.75) to (0,0.75);
      \draw[<-] (0.5,-0.75) to[out=up,in=up,looseness=0.8] (2,-0.75);
      \draw[->] (1,-0.75) \braidto (1.75,0) \braidto (1,0.75);
      \draw[<-] (1.5,-0.75) \braidto (2.25,0) \braidto (1.5,0.75);
      \draw[<-] (2.5,-0.75) to[out=up,in=down,looseness=0.8] (1.25,0) to[out=up,in=down,looseness=0.8] (2.5,0.75);
      \draw[->] (0.5,0.75) to[out=down,in=down,looseness=0.8] (2,0.75);
    \end{tikzpicture}
    +
    \begin{tikzpicture}[anchorbase]
      \draw[->] (0,-0.75) \braidto (1,0.75);
      \draw[<-] (1.5,-0.75) \braidto (2.5,0.75);
      \draw[->] (2,-0.75) \braidto (0,0.75);
      \draw[<-] (2.5,-0.75) \braidto (0.5,0.75);
      \draw[->] (1.5,0.75) to (1.5,0.65) to[out=down,in=down,looseness=1.5] (2,0.65) to (2,0.75);
      \draw[->] (1,-0.75) \braidto (1.7,0.1) to[out=up,in=up,looseness=1.5] (1.2,0.1) to[out=down,in=up] (0.5,-0.75);
    \end{tikzpicture}
    +
    \begin{tikzpicture}[anchorbase]
      \draw[->] (0,-0.75) to[out=up,in=240] (2,0.75);
      \draw[<-] (0.5,-0.75) to[out=60,in=down] (2.5,0.75);
      \draw[->] (1,-0.75) \braidto (0,0.75);
      \draw[<-] (1.5,-0.75) to (1.5,-0.65) to[out=up,in=up,looseness=1.5] (2,-0.65) to (2,-0.75);
      \draw[<-] (2.5,-0.75) \braidto (1.5,0.75);
      \draw[->] (0.5,0.75) to (0.5,0.65) to[out=down,in=down,looseness=1.5] (1,0.65) to (1,0.75);
    \end{tikzpicture}
    +
    \begin{tikzpicture}[anchorbase]
      \draw[->] (0,-0.75) to (0,0.75);
      \draw[<-] (0.5,-0.75) to (0.5,-0.65) to[out=up,in=up,looseness=1.5] (1,-0.65) to (1,-0.75);
      \draw[<-] (1.5,-0.75) to (1.5,-0.65) to[out=up,in=up,looseness=1.5] (2,-0.65) to (2,-0.75);
      \draw[->] (0.5,0.75) to (0.5,0.65) to[out=down,in=down,looseness=1.5] (1,0.65) to (1,0.75);
      \draw[->] (1.5,0.75) to (1.5,0.65) to[out=down,in=down,looseness=1.5] (2,0.65) to (2,0.75);
      \draw[<-] (2.5,-0.75) to (2.5,0.75);
    \end{tikzpicture}
    \\
    &\underset{\cref{H3}}{\overset{\cref{H1}}{=}}
    \begin{tikzpicture}[anchorbase]
      \draw[->] (0,-0.75) to[out=up,in=240] (2,0.75);
      \draw[<-] (0.5,-0.75) to[out=60,in=down] (2.5,0.75);
      \draw[->] (1,-0.75) \braidto (2.1,0) \braidto (1,0.75);
      \draw[<-] (1.5,-0.75) to[out=60,in=down] (2.5,0) to[out=up,in=-60] (1.5,0.75);
      \draw[->] (2,-0.75) to[out=120,in=down] (0,0.75);
      \draw[<-] (2.5,-0.75) to[out=up,in=-60] (0.5,0.75);
    \end{tikzpicture}
    +
    \begin{tikzpicture}[anchorbase]
      \draw[->] (0,-0.75) to (0,0.75);
      \draw[->] (1,-0.75) to (1,0.75);
      \draw[<-] (1.5,-0.75) to (1.5,0.75);
      \draw[<-] (2.5,-0.75) to (2.5,0.75);
      \draw[<-] (0.5,-0.75) to[out=up,in=up] (2,-0.75);
      \draw[->] (0.5,0.75) to[out=down,in=down] (2,0.75);
    \end{tikzpicture}
    +
    \begin{tikzpicture}[anchorbase]
      \draw[->] (0,-0.75) \braidto (1,0.75);
      \draw[<-] (1.5,-0.75) \braidto (2.5,0.75);
      \draw[->] (2,-0.75) to[out=120,in=down] (0,0.75);
      \draw[<-] (2.5,-0.75) to[out=up,in=-60] (0.5,0.75);
      \draw[->] (1.5,0.75) to (1.5,0.65) to[out=down,in=down,looseness=1.5] (2,0.65) to (2,0.75);
      \draw[<-] (0.5,-0.75) to (0.5,-0.65) to[out=up,in=up,looseness=1.5] (1,-0.65) to (1,-0.75);
    \end{tikzpicture}
    +
    \begin{tikzpicture}[anchorbase]
      \draw[->] (0,-0.75) to[out=up,in=240] (2,0.75);
      \draw[<-] (0.5,-0.75) to[out=60,in=down] (2.5,0.75);
      \draw[->] (1,-0.75) \braidto (0,0.75);
      \draw[<-] (1.5,-0.75) to (1.5,-0.65) to[out=up,in=up,looseness=1.5] (2,-0.65) to (2,-0.75);
      \draw[<-] (2.5,-0.75) \braidto (1.5,0.75);
      \draw[->] (0.5,0.75) to (0.5,0.65) to[out=down,in=down,looseness=1.5] (1,0.65) to (1,0.75);
    \end{tikzpicture}
    +
    \begin{tikzpicture}[anchorbase]
      \draw[->] (0,-0.75) to (0,0.75);
      \draw[<-] (0.5,-0.75) to (0.5,-0.65) to[out=up,in=up,looseness=1.5] (1,-0.65) to (1,-0.75);
      \draw[<-] (1.5,-0.75) to (1.5,-0.65) to[out=up,in=up,looseness=1.5] (2,-0.65) to (2,-0.75);
      \draw[->] (0.5,0.75) to (0.5,0.65) to[out=down,in=down,looseness=1.5] (1,0.65) to (1,0.75);
      \draw[->] (1.5,0.75) to (1.5,0.65) to[out=down,in=down,looseness=1.5] (2,0.65) to (2,0.75);
      \draw[<-] (2.5,-0.75) to (2.5,0.75);
    \end{tikzpicture}
    \ .
  \end{align*}
  Similarly,
  \begin{multline*}
    \Psi_t(1_1 \otimes s) \circ \Psi_t(s \otimes 1_1) \circ \Psi_t(1_1 \otimes s) \\
    =
    \begin{tikzpicture}[anchorbase]
      \draw[->] (0,-0.75) to[out=up,in=240] (2,0.75);
      \draw[<-] (0.5,-0.75) to[out=60,in=down] (2.5,0.75);
      \draw[->] (1,-0.75) to[out=120,in=down] (0,0) to[out=up,in=240] (1,0.75);
      \draw[<-] (1.5,-0.75) \braidto (0.4,0) \braidto (1.5,0.75);
      \draw[->] (2,-0.75) to[out=120,in=down] (0,0.75);
      \draw[<-] (2.5,-0.75) to[out=up,in=-60] (0.5,0.75);
    \end{tikzpicture}
    +
    \begin{tikzpicture}[anchorbase]
      \draw[->] (0,-0.75) to (0,0.75);
      \draw[->] (1,-0.75) to (1,0.75);
      \draw[<-] (1.5,-0.75) to (1.5,0.75);
      \draw[<-] (2.5,-0.75) to (2.5,0.75);
      \draw[<-] (0.5,-0.75) to[out=up,in=up] (2,-0.75);
      \draw[->] (0.5,0.75) to[out=down,in=down] (2,0.75);
    \end{tikzpicture}
    +
    \begin{tikzpicture}[anchorbase]
      \draw[->] (0,-0.75) \braidto (1,0.75);
      \draw[<-] (1.5,-0.75) \braidto (2.5,0.75);
      \draw[->] (2,-0.75) to[out=120,in=down] (0,0.75);
      \draw[<-] (2.5,-0.75) to[out=up,in=-60] (0.5,0.75);
      \draw[->] (1.5,0.75) to (1.5,0.65) to[out=down,in=down,looseness=1.5] (2,0.65) to (2,0.75);
      \draw[<-] (0.5,-0.75) to (0.5,-0.65) to[out=up,in=up,looseness=1.5] (1,-0.65) to (1,-0.75);
    \end{tikzpicture}
    +
    \begin{tikzpicture}[anchorbase]
      \draw[->] (0,-0.75) to[out=up,in=240] (2,0.75);
      \draw[<-] (0.5,-0.75) to[out=60,in=down] (2.5,0.75);
      \draw[->] (1,-0.75) \braidto (0,0.75);
      \draw[<-] (1.5,-0.75) to (1.5,-0.65) to[out=up,in=up,looseness=1.5] (2,-0.65) to (2,-0.75);
      \draw[<-] (2.5,-0.75) \braidto (1.5,0.75);
      \draw[->] (0.5,0.75) to (0.5,0.65) to[out=down,in=down,looseness=1.5] (1,0.65) to (1,0.75);
    \end{tikzpicture}
    +
    \begin{tikzpicture}[anchorbase]
      \draw[->] (0,-0.75) to (0,0.75);
      \draw[<-] (0.5,-0.75) to (0.5,-0.65) to[out=up,in=up,looseness=1.5] (1,-0.65) to (1,-0.75);
      \draw[<-] (1.5,-0.75) to (1.5,-0.65) to[out=up,in=up,looseness=1.5] (2,-0.65) to (2,-0.75);
      \draw[->] (0.5,0.75) to (0.5,0.65) to[out=down,in=down,looseness=1.5] (1,0.65) to (1,0.75);
      \draw[->] (1.5,0.75) to (1.5,0.65) to[out=down,in=down,looseness=1.5] (2,0.65) to (2,0.75);
      \draw[<-] (2.5,-0.75) to (2.5,0.75);
    \end{tikzpicture}
    \ .
  \end{multline*}
  Hence it follows from \cref{genbraid} that $\Psi_t$ preserves the second relation in \cref{P2}.

  To verify the first relation in \cref{P3}, we compute
  \[
    \Psi_t(s) \circ (1_1 \otimes \eta)
    =\
    \begin{tikzpicture}[anchorbase]
      \draw[->] (0,-0.5) \braidto (1,0.5);
      \draw[<-] (0.5,-0.5) \braidto (1.5,0.5);
      \draw[<-] (0,0.5) to[out=down,in=up] (1,-0.2) to[out=down,in=down,looseness=1.5] (1.4,-0.2) to[out=up,in=-60] (0.5,0.5);
    \end{tikzpicture}
    \ +\
    \begin{tikzpicture}[anchorbase]
      \draw[->] (0,-0.5) to (0,0.5);
      \draw[<-] (0.5,-0.5) to (0.5,-0.2) to[out=up,in=up,looseness=1.5] (1,-0.2) to[out=down,in=down,looseness=1.5] (1.5,-0.2) to (1.5,0.5);
      \draw[->] (0.5,0.5) to (0.5,0.4) to[out=down,in=down,looseness=1.5] (1,0.4) to (1,0.5);
    \end{tikzpicture}
    \ \underset{\cref{H3}}{\overset{\cref{H1}}{=}}\
    \begin{tikzpicture}[anchorbase]
      \draw[<-] (0,0.5) to (0,0.4) to[out=down,in=down,looseness=1.5] (0.5,0.4) to (0.5,0.5);
      \draw[->] (1,-0.5) to (1,0.5);
      \draw[<-] (1.5,-0.5) to (1.5,0.5);
    \end{tikzpicture}
    = \Psi_t(\eta \otimes 1_1).
  \]

  To verify the second relation in \cref{P3}, we first compute
  \begin{multline*}
    \Psi_t(s \otimes 1_1) \circ \Psi_t(1_1 \otimes s)
    \\
    =
    \begin{tikzpicture}[anchorbase]
      \draw[<-] (2.5,0) to[out=up,in=-60] (0.5,1);
      \draw[->] (2,0) to[out=120,in=down] (0,1);
      \draw[<-] (1.5,0) \braidto (2.5,1);
      \draw[->] (1,0) \braidto (2,1);
      \draw[<-] (0.5,0) \braidto (1.5,1);
      \draw[->] (0,0) \braidto (1,1);
    \end{tikzpicture}
    \ +\
    \begin{tikzpicture}[anchorbase]
      \draw[<-] (2.5,0) to (2.5,1);
      \draw[->] (2,0) to (2,0.1) to[out=up,in=up,looseness=1.5] (1.5,0.1) to (1.5,0);
      \draw[->] (1,0) \braidto (0,1);
      \draw[<-] (0.5,0) \braidto (1.5,1);
      \draw[->] (0,0) \braidto (1,1);
      \draw[<-] (2,1) to[out=down,in=down] (0.5,1);
    \end{tikzpicture}
    \ +\
    \begin{tikzpicture}[anchorbase]
      \draw[<-] (2.5,0) \braidto (1.5,1);
      \draw[->] (2,0) to[out=up,in=up] (0.5,0);
      \draw[<-] (1.5,0) \braidto (2.5,1);
      \draw[->] (1,0) \braidto (2,1);
      \draw[->] (0,0) to (0,1);
      \draw[<-] (1,1) to (1,0.9) to[out=down,in=down,looseness=1.5] (0.5,0.9) to (0.5,1);
    \end{tikzpicture}
    \ +\
    \begin{tikzpicture}[anchorbase]
      \draw[->] (0,0) to (0,1);
      \draw[<-] (0.5,0) to (0.5,0.1) to[out=up,in=up,looseness=1.5] (1,0.1) to (1,0);
      \draw[<-] (1.5,0) to (1.5,0.1) to[out=up,in=up,looseness=1.5] (2,0.1) to (2,0);
      \draw[<-] (2.5,0) to (2.5,1);
      \draw[->] (0.5,1) to (0.5,0.9) to[out=down,in=down,looseness=1.5] (1,0.9) to (1,1);
      \draw[->] (1.5,1) to (1.5,0.9) to[out=down,in=down,looseness=1.5] (2,0.9) to (2,1);
    \end{tikzpicture}
    \ .
  \end{multline*}
  Then, using \cref{H3}, we have
  \[
    \Psi_t(1_1 \otimes \mu) \circ \Psi_t(s \otimes 1_1) \circ \Psi_t(1_1 \otimes s)
    =
    \begin{tikzpicture}[anchorbase]
      \draw[->] (0,-0.5) \braidto (1.5,0.5);
      \draw[<-] (1.5,-0.5) \braidto (2,0.5);
      \draw[->] (2,-0.5) to[out=120,in=down] (0.5,0.5);
      \draw[<-] (2.5,-0.5) to[out=up,in=-60] (1,0.5);
      \draw[<-] (0.5,-0.5) to (0.5,-0.4) to[out=up,in=up,looseness=1.5] (1,-0.4) to (1,-0.5);
    \end{tikzpicture}
    +
    \begin{tikzpicture}[anchorbase]
      \draw[->] (0,-0.5) \braidto (0.5,0.5);
      \draw[<-] (0.5,-0.5) to (0.5,-0.4) to[out=up,in=up,looseness=1.5] (1,-0.4) to (1,-0.5);
      \draw[<-] (1.5,-0.5) to (1.5,-0.4) to[out=up,in=up,looseness=1.5] (2,-0.4) to (2,-0.5);
      \draw[->] (1,0.5) to (1,0.4) to[out=down,in=down,looseness=1.5] (1.5,0.4) to (1.5,0.5);
      \draw[<-] (2.5,-0.5) \braidto (2,0.5);
    \end{tikzpicture}
    = \Psi_t(s) \circ \Psi_t(\mu \otimes 1_1).
  \]
  The proofs of the second and third relations in \cref{P3} are analogous.

  Finally, to verify the relations \cref{P3}, we compute
  \begin{gather*}
    \Psi_t(\mu) \circ \Psi_t(s)
    =
    \begin{tikzpicture}[anchorbase]
      \draw[->] (0,-0.7) \braidto (1,0.2) to[out=up,in=up,looseness=1.5] (0.5,0.2) to[out=down,in=up] (1.5,-0.7);
      \draw[<-] (0.5,-0.7) \braidto (1.3,0) \braidto (1,0.7);
      \draw[->] (1,-0.7) \braidto (0.2,0) \braidto (0.5,0.7);
    \end{tikzpicture}
    \ +\
    \begin{tikzpicture}[anchorbase]
      \draw[->] (0,-0.7) \braidto (0.5,0.7);
      \draw[<-] (0.5,-0.7) to (0.5,-0.6) to[out=up,in=up,looseness=1.5] (1,-0.6) to (1,-0.7);
      \draw[<-] (1.5,-0.6) \braidto (1,0.7);
      \draw[->] (1,0) arc(0:360:0.2);
    \end{tikzpicture}
    \ \overset{\cref{H3}}{=} \
    \begin{tikzpicture}[anchorbase]
      \draw[->] (0,-0.7) \braidto (0.5,0.7);
      \draw[<-] (0.5,-0.7) to (0.5,-0.6) to[out=up,in=up,looseness=1.5] (1,-0.6) to (1,-0.7);
      \draw[<-] (1.5,-0.6) \braidto (1,0.7);
    \end{tikzpicture}
    \ = \Psi_t(\mu),
    \\
    \Psi_t(\mu) \circ \Psi_t(\parcop)
    =
    \begin{tikzpicture}[anchorbase]
      \draw[->] (-0.25,-0.7) \braidto (-0.4,0) \braidto (-0.25,0.7);
      \draw[<-] (0.25,-0.7) \braidto (0.4,0) \braidto (0.25,0.7);
      \draw[->] (0.2,0) arc(0:360:0.2);
    \end{tikzpicture}
    \ \overset{\cref{H3}}{=}\
    \begin{tikzpicture}[anchorbase]
      \draw[->] (-0.25,-0.7) to (-0.25,0.7);
      \draw[<-] (0.25,-0.7) to (0.25,0.7);
    \end{tikzpicture}
    \ = \Psi_t(1_1),
    \\
    \Psi_t(\varepsilon) \circ \Psi_t(\eta)
    =
    \begin{tikzpicture}[anchorbase]
      \draw[->] (0.3,0) arc(360:0:0.3);
    \end{tikzpicture}
    \overset{\cref{H4}}{=} t 1_\one
    = \Psi_t (t 1_0). \qedhere
  \end{gather*}
\end{proof}

As an example that will be used later, we compute
\begin{equation} \label{square}
  \Psi_t
  \left(
    \begin{tikzpicture}[anchorbase]
      \pd{0,0};
      \pd{0.5,0};
      \pd{0,0.5};
      \pd{0.5,0.5};
      \draw (0,0) to (0.5,0) to (0.5,0.5) to (0,0.5) to (0,0);
    \end{tikzpicture}
  \right)
  =
  \Psi_t
  \left(
    \begin{tikzpicture}[anchorbase]
      \pd{0,0};
      \pd{0.5,0};
      \pd{0.25,0.5};
      \pd{0,1};
      \pd{0.5,1};
      \draw (0,0) \braidto (0.25,0.5) \braidto (0,1);
      \draw (0.5,0) \braidto (0.25,0.5) \braidto (0.5,1);
    \end{tikzpicture}
  \right)
  =
  \begin{tikzpicture}[anchorbase]
    \draw[->] (1.2,0) -- (1.2,1);
    \draw[->] (1.5,1) -- (1.5,0.9) arc (180:360:.25) -- (2,1);
    \draw[<-] (1.5,0) -- (1.5,0.1) arc (180:0:.25) -- (2,0);
    \draw[<-] (2.3,0) -- (2.3,1);
  \end{tikzpicture}
  \ .
\end{equation}
Note that this is the second term in $\Psi_t(s)$ from \cref{functordef}.

\begin{rem} \label{bike}
  There are two natural ways to enlarge the codomain of the functor $\Psi_t$ to the entire Heisenberg category $\Heis$ (or a suitable quotient), rather than the category $\Peis(t)$.  The obstacle to this is that the clockwise bubble is not central in $\Heis$ and so the relation \cref{H4} is not well behaved there.  We continue to suppose that $\kk$ is a commutative ring.
  \begin{enumerate}
    \item \label{wheel} We can define $\Par$ to be the $\kk$-linear \emph{partition category with bubbles}, which has the same presentation as in \cref{Ppresent}, but without the last relation in \cref{P4}.  Free floating blocks (i.e.\ blocks not containing any vertices at the top or bottom of a diagram) are strictly central ``bubbles''.
        \details{
          They are strictly central since
          \[
            \begin{tikzpicture}[anchorbase]
              \pd{0,0.5};
              \pd{0.5,0};
              \pd{0.5,1};
              \draw (0,0.25) to (0,0.75);
              \draw (0.5,0) to (0.5,1);
            \end{tikzpicture}
            \ =\
            \begin{tikzpicture}[anchorbase]
              \pd{0,0.5};
              \pd{0,1};
              \pd{0,1.5};
              \pd{0.5,0};
              \pd{0.5,0.5};
              \pd{0.5,1};
              \pd{0.5,1.5};
              \pd{0.5,2};
              \draw (0,0.25) to (0,0.5) \braidto (0.5,1) \braidto (0,1.5) to (0,1.75);
              \draw (0.5,0) to (0.5,0.5) \braidto (0,1) \braidto (0.5,1.5) to (0.5,2);
            \end{tikzpicture}
            \ =\
            \begin{tikzpicture}[anchorbase]
              \pd{0.5,0.5};
              \pd{0,0};
              \pd{0,1};
              \draw (0.5,0.25) to (0.5,0.75);
              \draw (0,0) to (0,1);
            \end{tikzpicture}
            \ .
          \]
        }
        The category $\Par(t)$ is obtained from $\Par$ by specializing the bubble at $t$.  Then we have a $\kk$-linear monoidal functor $\Par \to \Heis$ (factoring through $\Peis$) mapping the bubble of $\Par$ to the clockwise bubble
        \[
          \begin{tikzpicture}[anchorbase]
            \draw[->] (0,0.3) arc(450:90:0.3);
          \end{tikzpicture}
          \ .
        \]
        This is equivalent to considering $\Par(t)$ over the ring $\kk[t]$ and $\Heis$ over $\kk$ and viewing $\Psi_t$ as a $\kk$-linear monoidal functor $\Par(t) \to \Heis$ with
        \[
          t \mapsto
          \begin{tikzpicture}[anchorbase]
            \draw[->] (0,0.3) arc(450:90:0.3);
          \end{tikzpicture}
          \ .
        \]
        (Then $t$ is the ``bubble'' in the partition category.)  We refer to this setting by saying that $t$ is \emph{generic}.

    \item If $t \in \kk$, let $\cI$ denote the left tensor ideal of $\Heis$ generated by
      \[
        \begin{tikzpicture}[anchorbase]
          \draw[->] (0,0.3) arc(450:90:0.3);
        \end{tikzpicture}
        - t 1_\one.
      \]
      Then $\Psi_t$ induces a $\kk$-linear functor
      \[
        \Par(t) \to \Heis/\cI.
      \]
      Note, however, that this induced functor is no longer monoidal.  Rather, it should be thought of as an action of $\Par(t)$ on the quotient $\Heis/\cI$.
    \end{enumerate}
\end{rem}

As noted in \cref{sec:partition}, the partition category is the free $\kk$-linear symmetric monoidal category generated by an $t$-dimensional special commutative Frobenius object.  Thus, \cref{functordef} implies that $\uparrow \downarrow$, together with certain morphisms, is a special commutative Frobenius object in the Heisenberg category.  Note, however, that neither the Heisenberg category nor $\Peis(t)$ is symmetric monoidal.

\section{Actions and faithfullness\label{sec:faithful}}

Consider the standard embedding of $S_{n-1}$ in $S_n$, and hence of $\kk S_{n-1}$ in $\kk S_n$.  Recall that we adopt the convention that $\kk S_n = \kk$ when $n=0$.  We have the natural induction and restriction functors
\[
  \Ind_{n-1}^n \colon S_{n-1}\md \to S_n\md,\qquad
  \Res_{n-1}^n \colon S_n\md \to S_{n-1}\md.
\]

If we let $B$ denote $\kk S_n$, considered as an $(S_n, S_{n-1})$-bimodule, then we have
\[
  \Ind_{n-1}^n \Res_{n-1}^n (M) = B \otimes_{n-1} M,\quad M \in S_n\md,
\]
where we recall that $\otimes_{n-1}$ denotes the tensor product over $\kk S_{n-1}$.  We will use the unadorned symbol $\otimes$ to denote tensor product over $\kk$.  As before, we denote the trivial one-dimensional $S_n$-module by $\mathbf{1}_n$.

Recall the coset representatives $g_i \in S_n$ defined in \cref{gi-def}.  In particular, we have
\begin{equation} \label{hungry}
  g_i^{-1} g_j \in S_{n-1} \iff i=j.
\end{equation}

Let $V = \kk^n$ be the permutation $S_n$-module with basis $v_1,\dotsc,v_n$.  Then we have
\begin{equation} \label{yeehah}
  B \otimes_{n-1} \mathbf{1}_{n-1} = \Ind_{n-1}^n(\mathbf{1}_{n-1}) \cong V
  \quad \text{as $S_n$-modules}.
\end{equation}
Furthermore, the elements $g_i \otimes_{n-1} 1$, $1 \le i \le n$, form a basis of $B \otimes_{n-1} \mathbf{1}_{n-1}$ and the isomorphism \cref{yeehah} is given explicitly by
\[
  B \otimes_{n-1} \mathbf{1}_{n-1} \xrightarrow{\cong} V,\quad
  g_i \otimes_{n-1} 1 \mapsto v_i = g_i v_n.
\]

More generally, define
\[
  B^k := \underbrace{B \otimes_{n-1} B \otimes_{n-1} \dotsb \otimes_{n-1} B}_{k \text{ factors}}.
\]
Then we have an isomorphism of $S_n$-modules
\begin{align*}
  \beta_k \colon V^{\otimes k} &\xrightarrow{\cong} B^k \otimes_{n-1} \mathbf{1}_{n-1}, \\
  v_{i_k} \otimes \dotsb \otimes v_{i_1} &\mapsto g_{i_k} \otimes g_{i_k}^{-1} g_{i_{k-1}} \otimes \dotsb \otimes g_{i_2}^{-1} g_{i_1} \otimes 1,\quad
  1 \le i_1,\dotsc,i_k \le n,
\end{align*}
extended by linearity, with inverse map
\begin{align*}
  \beta_k^{-1} \colon
  B^k \otimes_{n-1} \mathbf{1}_{n-1} &\xrightarrow{\cong} V^{\otimes k}, \\
  \pi_k \otimes \dotsb \otimes \pi_1 \otimes 1
  &\mapsto (\pi_k v_n) \otimes (\pi_k \pi_{k-1} v_n) \otimes \dotsm \otimes (\pi_k \dotsm \pi_1 v_n), \quad
  \pi_1,\dotsc,\pi_k \in S_n,
\end{align*}
extended by linearity.

\begin{theo} \label{actcom}
  Fix $n \in \N$, and recall the following functors from \cref{Phidef}, \cref{Omegadef}, and \cref{functordef}:
  \begin{equation}
    \begin{tikzcd}[column sep=2cm]
      \Par(n) \arrow[r,"\Psi_n"] \arrow[rd, swap,"\Phi_n"] &
      \Peis(n) \arrow[d,"\Omega_n"] \\
      & S_n\md
    \end{tikzcd}
  \end{equation}
  The morphisms $\beta_k$, $k \in \N$, give a natural isomorphism of functors $\Omega_n \circ \Psi_n \cong \Phi_n$.
\end{theo}

\begin{proof}
  Since the $\beta_k$ are isomorphisms, it suffices to verify that they define a natural transformation.  For this, we check the images of a set of generators of $\Par(n)$.  Since the functor $\Omega_n$ is not monoidal, we need to consider generators of $\Par(n)$ as a $\kk$-linear category.  Such a set of generators is given by
  \[
    1_k \otimes x \otimes 1_j,\quad
    k,j \in \N,\ x \in \{\mu,\parcop,s,\eta,\epsilon\}.
  \]
  See, for example, \cite[Th.~5.2]{Liu18}.

  Let $j \in \{1,2,\dotsc,n-1\}$.  We compute that
  \[
    \beta_{k-1}^{-1} \circ \left( \Omega_n \circ \Psi_n \left( 1_{k-j-1} \otimes \mu \otimes 1_{j-1} \right) \right) \circ \beta_k
    \colon V^{\otimes k} \to V^{\otimes {k-1}}
  \]
  is the $S_n$-module map given by
  \begin{align*}
    v_{i_k} \otimes \dotsb \otimes v_{i_1}
    &\mapsto g_{i_k} \otimes g_{i_k}^{-1} g_{i_{k-1}} \otimes \dotsb \otimes g_{i_2}^{-1} g_{i_1} \otimes 1
    \\
    &\overset{\mathclap{\cref{hungry}}}{\mapsto}\ \delta_{i_j,i_{j+1}} g_{i_k} \otimes g_{i_k}^{-1} g_{i_{k-1}} \otimes \dotsb \otimes g_{i_{j+3}}^{-1} g_{i_{j+2}} \otimes g_{i_{j+2}}^{-1} g_{i_j} \otimes g_{i_j}^{-1} g_{i_{j-1}} \otimes \dotsb \otimes g_{i_2}^{-1} g_{i_1} \otimes 1
    \\
    &\mapsto \delta_{i_j,i_{j+1}} v_{i_k} \otimes \dotsb \otimes v_{i_{j+2}} \otimes v_{i_j} \otimes \dotsb \otimes v_{i_1}.
  \end{align*}
  This is precisely the map $\Phi_n (1_{k-j-1} \otimes \mu \otimes 1_{j-1})$.

  Similarly, we compute that
  \[
    \beta_{k+1}^{-1} \circ \left( \Omega_n \circ \Psi_n \left( 1_{k-j} \otimes \parcop \otimes 1_{j-1} \right) \right) \circ \beta_k
    \colon V^{\otimes k} \to V^{\otimes {k+1}}
  \]
  is the $S_n$-module map given by
  \begin{align*}
    v_{i_k} \otimes \dotsb \otimes v_{i_1}
    &\mapsto g_{i_k} \otimes g_{i_k}^{-1} g_{i_{k-1}} \otimes \dotsb \otimes g_{i_2}^{-1} g_{i_1} \otimes 1
    \\
    &\mapsto g_{i_k} \otimes g_{i_k}^{-1} g_{i_{k-1}} \otimes \dotsb \otimes g_{i_{j+1}}^{-1} g_{i_j} \otimes 1 \otimes g_{i_j}^{-1} g_{i_{j-1}} \otimes \dotsb \otimes g_{i_2}^{-1} g_{i_1} \otimes 1 \\
    &= g_{i_k} \otimes g_{i_k}^{-1} g_{i_{k-1}} \otimes \dotsb \otimes g_{i_{j+1}}^{-1} g_{i_j} \otimes g_{i_j}^{-1} g_{i_j} \otimes g_{i_j}^{-1} g_{i_{j-1}} \otimes \dotsb \otimes g_{i_2}^{-1} g_{i_1} \otimes 1 \\
    &\mapsto v_{i_k} \otimes \dotsb \otimes v_{i_{j+1}} \otimes v_{i_j} \otimes v_{i_j} \otimes v_{i_{j-1}} \otimes \dotsb \otimes v_{i_1}.
  \end{align*}
  This is precisely the map $\Phi_n (1_{k-j} \otimes \parcop \otimes 1_{j-1})$.

  Now let $j \in \{1,\dotsc,n\}$.  We compute that
  \[
    \beta_{k+1}^{-1} \circ \left( \Omega_n \circ \Psi_n \left( 1_{k-j} \otimes \eta \otimes 1_j \right) \right) \circ \beta_k
    \colon V^{\otimes k} \to V^{\otimes {k+1}}
  \]
  is the map
  \begin{align*}
    v_{i_k} &\otimes \dotsb \otimes v_{i_1}
    \mapsto g_{i_k} \otimes g_{i_k}^{-1} g_{i_{k-1}} \otimes \dotsb \otimes g_{i_2}^{-1} g_{i_1} \otimes 1
    \\
    &\mapsto \sum_{m=1}^n g_{i_k} \otimes g_{i_k}^{-1} g_{i_{k-1}} \otimes \dotsb \otimes g_{i_{j+2}}^{-1} g_{i_{j+1}} \otimes g_{j+1}^{-1} g_m \otimes g_m^{-1} g_j \otimes g_{i_j}^{-1} g_{i_{j-1}} \otimes \dotsb \otimes g_{i_2}^{-1} g_{i_1} \otimes 1 \\
    &\mapsto \sum_{m=1}^n v_{i_k} \otimes \dotsb \otimes v_{i_{j+1}} \otimes v_m \otimes v_{i_j} \otimes \dotsb \otimes v_{i_1}.
  \end{align*}
  This is precisely the map $\Phi_n (1_{k-j} \otimes \eta \otimes 1_j)$.

  We also compute that
  \[
    \beta_{k-1}^{-1} \circ \left( \Omega_n \circ \Psi_n \left( 1_{k-j} \otimes \varepsilon \otimes 1_{j-1} \right) \right) \circ \beta_k
    \colon V^{\otimes k} \to V^{\otimes {k-1}}
  \]
  is the map
  \begin{align*}
    v_{i_k} \otimes \dotsb \otimes v_{i_1}
    &\mapsto g_{i_k} \otimes g_{i_k}^{-1} g_{i_{k-1}} \otimes \dotsb \otimes g_{i_2}^{-1} g_{i_1} \otimes 1
    \\
    &\mapsto g_{i_k} \otimes g_{i_k}^{-1} g_{i_{k-1}} \otimes \dotsb \otimes g_{i_{j+2}}^{-1} g_{i_{j+1}} \otimes g_{i_{j+1}}^{-1} g_{i_{j-1}} \otimes g_{i_{j-1}}^{-1} g_{i_{j-2}} \otimes \dotsb \otimes g_{i_2}^{-1} g_{i_1} \otimes 1
    \\
    &\mapsto v_{i_k} \otimes \dotsb \otimes v_{i_{j+1}} \otimes v_{i_{j-1}} \otimes \dotsb \otimes v_{i_1}.
  \end{align*}
  This is precisely the map $\Phi_n (1_{k-j} \otimes \varepsilon \otimes 1_{j-1})$.

  It remains to consider the generator $s$.  Let $j \in \{1,2,\dotsc,n-1\}$.  Define the elements $x,y \in \End_{\Heis}(\uparrow \downarrow \uparrow \downarrow)$ by
  \begin{equation} \label{breakdown}
    x =
    \begin{tikzpicture}[anchorbase]
      \draw[->] (0,0) \braidto (1,1);
      \draw[<-] (0.5,0) \braidto (1.5,1);
      \draw[->] (1,0) \braidto (0,1);
      \draw[<-] (1.5,0) \braidto (0.5,1);
    \end{tikzpicture}
    \ ,\qquad
    y =
    \begin{tikzpicture}[anchorbase]
      \draw[->] (1.2,0) -- (1.2,1);
      \draw[->] (1.5,1) -- (1.5,0.9) arc (180:360:.25) -- (2,1);
      \draw[<-] (1.5,0) -- (1.5,0.1) arc (180:0:.25) -- (2,0);
      \draw[<-] (2.3,0) -- (2.3,1);
    \end{tikzpicture}
    \ .
  \end{equation}
  Note that
  \[
    x = x_3 \circ x_2 \circ x_1,
  \]
  where
  \[
    x_1 =
    \begin{tikzpicture}[anchorbase]
      \draw[->] (0,0) to (0,0.6);
      \draw[<-] (0.5,0) \braidto (1,0.6);
      \draw[->] (1,0) \braidto (0.5,0.6);
      \draw[<-] (1.5,0) to (1.5,0.6);
    \end{tikzpicture}
    \ ,\quad
    x_2 =
    \begin{tikzpicture}[anchorbase]
      \draw[->] (0,0) \braidto (0.5,0.6);
      \draw[->] (0.5,0) \braidto (0,0.6);
      \draw[<-] (1,0) \braidto (1.5,0.6);
      \draw[<-] (1.5,0) \braidto (1,0.6);
    \end{tikzpicture}
    \ ,\quad
    x_3 =
    \begin{tikzpicture}[anchorbase]
      \draw[->] (0,0) to (0,0.6);
      \draw[->] (0.5,0) \braidto (1,0.6);
      \draw[<-] (1,0) \braidto (0.5,0.6);
      \draw[<-] (1.5,0) to (1.5,0.6);
    \end{tikzpicture}
    \ .
  \]
  Suppose $i,j \in \{1,\dotsc,n\}$ and $h,h' \in \kk S_n$.  We first compute the action of $\Theta(x)$ and $\Theta(y)$ on the element
  \[
    \alpha = h g_i \otimes g_i^{-1} g_j \otimes g_j^{-1} h'
    \in (n)_{n-1}(n)_{n-1}(n).
  \]
  If $i = j$, then $g_i^{-1} g_j = 1$, and so $\Phi_n(x_1)(\alpha) = 0$.  Now suppose $i < j$.  Then we have
  \[
    g_i^{-1} g_j
    = s_{n-1} \dotsm s_i s_j \dotsm s_{n-1}
    = s_{j-1} \dotsm s_{n-2} s_{n-1} s_{n-2} \dotsm s_i.
  \]
  Hence
  \[
    \Theta(x_1)(\alpha)
    = h g_i s_{j-1} \dotsm s_{n-2} \otimes s_{n-2} \dotsm s_i g_j^{-1} h'
    \in (n)_{n-2}(n).
  \]
  Thus
  \[
    \Theta(x_2 \circ x_1)(\alpha)
    = h g_i s_{j-1} \dotsm s_{n-1} \otimes g_i^{-1} g_j^{-1} h'
    \in (n)_{n-2}(n),
  \]
  and so
  \begin{align*}
    \Theta(x)(\alpha)
    &= h g_i s_{j-1} \dotsm s_{n-1} \otimes s_{n-1} \otimes g_i^{-1} g_j^{-1} h' \\
    &= h g_j g_i s_{n-1} \otimes s_{n-1} \otimes s_{n-2} \dotsb s_{j-1} g_i^{-1} h' \\
    &= h g_j \otimes g_i s_{n-2} \dotsb s_{j-1} \otimes g_i^{-1} h' \\
    &= h g_j \otimes g_j^{-1} g_i \otimes g_i^{-1} h'.
  \end{align*}
  The case $i > j$ is similar,
  \details{
    Suppose $i > j$.  Then we have
    \[
      g_i^{-1} g_j
      = s_{n-1} \dotsm s_i s_j \dotsm s_{n-1}
      = s_j \dotsm s_{n-2} s_{n-1} s_{n-2} \dotsm s_{i-1}.
    \]
    Hence
    \[
      \Theta(x_1)(\alpha)
      = h g_i s_j \dotsm s_{n-2} \otimes s_{n-2} \dotsm s_{i-1} g_j^{-1} h'
      \in (n)_{n-2}(n).
    \]
    Thus
    \[
      \Theta(x_2 \circ x_1)(\alpha)
      = h g_i g_j \otimes s_{n-1} \dotsb s_{i-1} g_j^{-1} h'
      \in (n)_{n-2}(n),
    \]
    and so
    \begin{align*}
      \Theta(x)(\alpha)
      &= h g_i g_j \otimes s_{n-1} \otimes s_{n-1} \dotsb s_{i-1} g_j^{-1} h' \\
      &= h g_j s_{i-1} \dotsb s_{n-2} \otimes s_{n-1} \otimes s_{n-1} g_j^{-1} g_i^{-1} h' \\
      &= h g_j \otimes s_{i-1} \dotsb s_{n-2} g_j^{-1} \otimes g_i^{-1} h' \\
      &= h g_j \otimes g_j^{-1} g_i \otimes g_i^{-1} h'.
    \end{align*}
  }
  giving
  \[
    \Theta(x)(h g_i \otimes g_i^{-1} g_j \otimes g_j^{-1} h')
    =
    \begin{cases}
      0 & \text{if } i=j, \\
      h g_j \otimes g_j^{-1} g_i \otimes g_i^{-1} h' & \text{if } i \ne j.
    \end{cases}
  \]
  We also easily compute that
  \[
    \Theta(y)(h g_i \otimes g_i^{-1} g_j \otimes g_j^{-1} h')
    =
    \begin{cases}
      h g_i \otimes 1 \otimes g_i^{-1} h' & \text{if } i=j, \\
      0 & \text{if } i \ne j.
    \end{cases}
  \]
  Thus, for all $i,j \in \{1,\dotsc,n\}$, we have
  \begin{equation} \label{turkey}
    \Theta(x+y)(h g_i \otimes g_i^{-1} g_j \otimes g_j^{-1} h')
    = h g_j \otimes g_j^{-1} g_i \otimes g_i^{-1} h'.
  \end{equation}
  It now follows easily that
  \[
    \beta_k^{-1} \circ \left( \Omega_n \circ \Psi_n \left( 1_{k-j-1} \otimes s \otimes 1_{j-1} \right) \right) \circ \beta_k
     = \beta_k^{-1} \circ \Omega_n \left(1_{\uparrow \downarrow}^{\otimes (k-j-1)} \otimes (x+y) \otimes 1_{\uparrow \downarrow}^{\otimes (j-1)} \right) \circ \beta_k
  \]
  is the map given by
  \[
    v_{i_k} \otimes \dotsb \otimes v_{i_1}
    \mapsto v_{i_k} \otimes \dotsb \otimes v_{i_{j+2}} \otimes v_{i_j} \otimes v_{i_{j+1}} \otimes v_{i_{j-1}} \otimes \dotsb \otimes v_{i_1},
  \]
  which is precisely the map $\Phi_n \left( 1_{k-j-1} \otimes s \otimes 1_{j-1} \right)$.
\end{proof}

\begin{theo} \label{faithful}
  The functor $\Psi_t$ is faithful.
\end{theo}

We give here a proof under the assumption that $\kk$ is a commutative ring of characteristic zero.  The general case will be treated in \cref{appendix}.

\begin{proof}
  It suffices to show that given $k,\ell \ge 0$ the linear map
  \[
    \Psi_t(k,\ell) \colon \Hom_{\Par(t)}(k,\ell)
    \to\Hom_{\Peis}\big( (\uparrow\downarrow)^k, (\uparrow\downarrow)^\ell \big)
    \]
  is injective.  Consider $\Par(t)$ over $\kk[t]$ and suppose
  \[
    f = \sum_{i=1}^m a_i(t) f_i \in \ker \left( \Psi_t(k,\ell) \right)
  \]
  for some $a_i \in \kk[t]$ and partition diagrams $f_i$.  Choose $n \ge k + \ell$ and evaluate at $t=n$ to get
  \[
    f_n = \sum_{i=1}^m a_i(n) f_i \in \Hom_{\Par(n)}(k,\ell).
  \]
  (Here $\Par(n)$ is a $\kk$-linear category.)  \Cref{actcom} implies that $\Phi_n(f_n)=0$.  Then \cref{bee} implies $f_n=0$.  Since the partition diagrams form a basis for the morphisms spaces in $\Par(n)$, we have $a_i(n) = 0$ for all $i$.  Since this holds for all $n \ge k + \ell$, we have $a_i = 0$ for all $i$.  (Here we use that the characteristic of $\kk$ is zero.) Hence $f=0$ and so $\Psi_t$ is faithful.
\end{proof}

Since any faithful linear monoidal functor induces a faithful linear monoidal functor on additive Karoubi envelopes, we obtain the following corollary.

\begin{cor}
  The functor $\Psi_t$ induces a faithful linear monoidal functor from Deligne's category $\Rep(S_t)$ to the additive Karoubi envelope $\Kar(\Peis(t))$ of $\Peis(t)$.
\end{cor}

Note that $\Psi_t$ is \emph{not} full.  This follows immediately from the fact that the morphism spaces in $\Par(t)$ are finite-dimensional, while those in $\Peis$ are infinite-dimensional, as follows from the explicit basis described in \cite[Prop.~5]{Kho14} (see also \cite[Th.~6.4]{BSW18}).

\section{Grothendieck rings\label{sec:Groth}}

In this section, we assume that $\kk$ is a field of characteristic zero.  We consider $\Par(t)$ over the ground ring $\kk[t]$ and $\Heis$ over $\kk$.  We can then view $\Psi_t$ as a $\kk$-linear functor $\Par(t) \to \Heis$ as noted in \cref{bike}\cref{wheel}.

For an additive linear monoidal category $\cC$, we let $K_0(\cC)$ denote its split Grothendieck ring.  The multiplication in $K_0(\cC)$ is given by $[X] [Y] = [X \otimes Y]$, where $[X]$ denotes the class in $K_0(\cC)$ of an object $X$ in $\cC$.

Recall that Deligne's category $\Rep(S_t)$ is the additive Karoubi envelope $\Kar(\Par(t))$ of the partition category.  The additive monoidal functor $\Psi_t$ of \cref{functordef} induces a ring homomorphism
\begin{equation} \label{juice}
  [\Psi_t] \colon K_0(\Rep(S_t)) \to K_0(\Kar(\Heis)),\quad
  [\Psi_t]([X]) = [\Psi_t(X)].
\end{equation}
The main result of this section (\cref{finally}) is a precise description of this homomorphism.

Let $\cY$ denote the set of Young diagrams $\lambda = (\lambda_1,\lambda_2,\dotsc,\lambda_\ell)$, $\lambda_1 \ge \lambda_2 \ge \dotsb \ge \lambda_\ell >  0$.  (We avoid the terminology \emph{partition} here to avoid confusion with the partition category.)  For a Young diagram $\lambda \in \cY$, we let $|\lambda|$ denote its size (i.e.\ the sum of its parts).  Let $\Sym$ denote the ring of symmetric functions with integer coefficients.  Then $\Sym$ has a $\Z$-basis given by the Schur functions $s_\lambda$, $\lambda \in \cY$.  We have
\begin{equation} \label{Symp}
  \Sym_\Q := \Q \otimes_\Z \Sym \cong \Q[p_1,p_2,\dotsc]
  = \bigoplus_{\lambda \in \cY} \Q p_\lambda,
\end{equation}
where $p_n$ denotes the $n$-th power sum and $p_\lambda = p_{\lambda_1} \dotsm p_{\lambda_k}$ for a Young diagram $\lambda = (\lambda_1,\dotsc,\lambda_k)$.

The \emph{infinite-dimensional Heisenberg Lie algebra} $\fh$ is the Lie algebra over $\Q$ generated by $\{p_n^\pm, c : n \ge 1\}$ subject to the relations
\[
  [p_m^-,p_n^-] = [p_m^+, p_n^+] = [c, p_n^\pm] = 0,\quad
  [p_m^+, p_n^-] = \delta_{m,n} n c.
\]
The central reduction $U(\fh)/(c+1)$ of its universal enveloping algebra can also be realized as the \emph{Heisenberg double} $\Sym_\Q \#_\Q \Sym_\Q$ with respect to the bilinear Hopf pairing
\[
  \langle -, - \rangle \colon \Sym_\Q \times \Sym_\Q,\quad
  \langle p_m, p_n \rangle = \delta_{m,n} n.
\]
By definition, $\Sym_\Q \#_\Q \Sym_\Q$ is the vector space $\Sym_\Q \otimes_\Q \Sym_\Q$ with associative multiplication given by
\[
  (e \otimes f)(g \otimes h)
  = \sum_{(f),(g)} \langle f_{(1)}, g_{(2)} \rangle e g_{(1)} \otimes f_{(2)} h,
\]
where we use Sweedler notation for the usual coproduct on $\Sym_\Q$ determined by
\begin{equation} \label{burrito}
  p_n \mapsto p_n \otimes 1 + 1 \otimes p_n.
\end{equation}
Comparing the coefficients appearing in \cite[Th.~5.3]{Ber17} to \cite[(2.2)]{Ber17}, we see that the pairing of two complete symmetric functions is an integer.  (Note that our $p_n^\pm$ are denoted $p_n^\mp$ in \cite{Ber17}.)  We can therefore restrict $\langle -, - \rangle$ to obtain a biadditive form $\langle -, - \rangle \colon \Sym \otimes_\Z \Sym \to \Z$.  The corresponding Heisenberg double
\[
  \rHeis := \Sym \#_\Z \Sym
\]
is a natural $\Z$-form for $U(\fh)/(c+1) \cong \Sym_\Q \#_\Q \Sym_\Q$.  For $f \in \Sym$ we let $f^-$ and $f^+$ denote the elements $f \otimes 1$ and $1 \otimes f$ of $\rHeis$, respectively.

Recall the algebra homomorphisms \cref{garage,SunnyD}, which we use to view elements of $\kk S_k$ as endomorphisms in the partition and Heisenberg categories.  In particular, the homomorphisms \cref{SunnyD} induce a natural algebra homomorphism
\begin{equation} \label{twostep}
  \kk S_k \otimes_\kk \kk S_k \to \End_\Heis(\uparrow^k \downarrow^k).
\end{equation}
We will use this homomorphism to view elements of $\kk S_k \otimes \kk S_k$ as elements of $\End_\Heis(\uparrow^k \downarrow^k)$.

One can deduce explicit presentations of $\rHeis$ (see \cite[\S5]{Ber17} and \cite[Appendix~A]{LRS18}), but we will not need such presentations here.  Important for our purposes is that
\[
  s_\lambda^+ s_\mu^-,\quad \lambda,\mu \in \cY,
\]
is a $\Z$-basis for $\rHeis$, and that there is an isomorphism of rings
\begin{equation} \label{K0Heis}
  \rHeis \xrightarrow{\cong} K_0(\Kar(\Heis)), \quad
  s_\lambda^+ s_\mu^- \mapsto [(\uparrow^{|\lambda|} \downarrow^{|\mu|}, e_\lambda \otimes e_\mu)],\quad
  \lambda,\mu \in \cY,
\end{equation}
where $e_\lambda$ is the Young symmetrizer corresponding to the Young diagram $\lambda$.  We adopt the convention that $e_\varnothing = 1$ and $s_\varnothing = 1$, where $\varnothing$ denotes the empty Young diagram of size $0$.  The isomorphism \cref{K0Heis} was conjectured in \cite[Conj.~1]{Kho14} and proved in \cite[Th.~1.1]{BSW18}.  Via the isomorphism \cref{K0Heis}, we will identify $K_0(\Kar(\Heis))$ and $\rHeis$ in what follows.

Recall that there is an isomorphism of Hopf algebras
\begin{equation} \label{pink}
  \bigoplus_{n=0}^\infty K_0(S_n\md) \cong \Sym,\quad [\kk S_n e_\lambda] \mapsto s_\lambda.
\end{equation}
The product on $\bigoplus_{n=0}^\infty K_0(S_n\md)$ is given by
\[
  [M] \cdot [N] = \left[ \Ind_{S_m \times S_n}^{S_{m+n}} (M \boxtimes N) \right],\quad
  M \in S_m\md,\ N \in S_n\md,
\]
while the coproduct \cref{burrito} is given by
\[
  [K] \mapsto \bigoplus_{n+m=k} \left[ \Res^{S_k}_{S_m \times S_n} K \right],\quad K \in S_k\md.
\]

In addition to the coproduct \cref{burrito}, there is another well-studied coproduct on $\Sym_\Q$, the \emph{Kronecker coproduct}, which is given by
\[
  \KDelt \colon \Sym_\Q \to \Sym_\Q \otimes_\Q \Sym_\Q,\quad
  \KDelt(p_\lambda) = p_\lambda \otimes p_\lambda.
\]
It is dual to the Kronecker (or internal) product on $\Sym_\Q$.  Restriction to $\Sym$ gives a coproduct
\begin{equation} \label{milk}
  \KDelt \colon \Sym \to \Sym \otimes_\Z \Sym.
\end{equation}
The fact that the restriction of $\KDelt$ to $\Sym$ lands in $\Sym \otimes_\Z \Sym$ is implied by the following categorical interpretation of the Kronecker coproduct.  The diagonal embedding $S_n \to S_n \times S_n$ extends by linearity to an injective algebra homomorphism
\begin{equation} \label{diagonal}
  d \colon \kk S_n \to \kk S_n \otimes_\kk \kk S_n.
\end{equation}
Under the isomorphism \cref{pink}, the functor
\[
  S_n\md \to (S_n \times S_n)\md,\quad M \mapsto \Ind_{S_n}^{S_n \times S_n}(M),
\]
corresponds precisely to $\KDelt$ after passing to Grothendieck groups.  (See \cite{Lit56}.)

Now view the Kronecker coproduct as a linear map
\begin{equation} \label{chariot}
  \KDelt \colon \Sym \to \Sym \#_\Z \Sym = \rHeis.
\end{equation}
It is clear that the map \cref{milk} is a ring homomorphism with the product ring structure on $\Sym \otimes_\Z \Sym$.  In fact, it turns out that we also have the following.

\begin{lem}
  The map \cref{chariot} is an injective ring homomorphism.
\end{lem}

\begin{proof}
  We prove the result over $\Q$; then the statement follows by restriction to $\Sym$.  By \cref{Symp}, it suffices to prove that $\KDelt(p_n)$ and $\KDelt(p_m)$ commute for $n,m \in \N$.  Since, for $n \ne m$,
  \[
    \KDelt(p_n) \KDelt(p_m)
    = p_n^+ p_n^- p_m^+ p_m^-
    = p_m^+ p_m^- p_n^+ p_n^-
    = \KDelt(p_m) \KDelt(p_n),
  \]
  we see that $\KDelt$ is a ring homomorphism.  It is clear that it is injective.
\end{proof}

Our first step in describing the map \cref{juice} is to decompose the objects $(\uparrow \downarrow)^k$ appearing in the image of $\Psi_t$.  Recall that the Stirling number of the second kind $\stirling{k}{\ell}$, $k, \ell \in \N$, counts the number of ways to partition a set of $k$ labelled objects into $\ell$ nonempty unlabelled subsets.  These numbers are given by
\[
  \stirling{k}{\ell} = \frac{1}{\ell!} \sum_{i=0}^\ell (-1)^i \binom{\ell}{i} (\ell-i)^k
\]
and are determined by the recursion relation
\[
  \stirling{k+1}{\ell} = \ell \stirling{k}{\ell} + \stirling{k}{\ell-1}
  \quad \text{with} \quad
  \stirling{0}{0}=1
  \quad \text{and} \quad
  \stirling{k}{0} = \stirling{0}{k} = 0,\ k > 0.
\]

\begin{lem}
  In $\Heis$, we have
  \begin{equation} \label{jungle}
    (\uparrow \downarrow)^k
    \cong\ \bigoplus_{\ell=1}^k (\uparrow^\ell \downarrow^\ell)^{\oplus \stirling{k}{\ell}}.
  \end{equation}
  In particular, since $\stirling{k}{k}=1$, the summand $\uparrow^k \downarrow^k$ appears with multiplicity one.
\end{lem}

\begin{proof}
  First note that repeated use of the isomorphism \cref{key} gives
  \begin{equation} \label{mail}
    \uparrow \downarrow \uparrow^k \downarrow^k
    \ \cong\ \uparrow^{k+1} \downarrow^{k+1} \oplus (\uparrow^k \downarrow^k)^{\oplus k}.
  \end{equation}
  We now prove the lemma by induction on $k$.  The case $k=1$ is immediate.  Suppose the result holds for some $k \ge 1$.  Then we have
  \begin{multline*}
    (\uparrow \downarrow)^{k+1}
    \cong (\uparrow \downarrow) \left( \bigoplus_{\ell=1}^k (\uparrow^\ell \downarrow^\ell)^{\oplus \stirling{k}{\ell}} \right)
    \overset{\cref{mail}}{\cong} \bigoplus_{\ell=1}^k \left( (\uparrow^{\ell+1} \downarrow^{\ell+1})^{\oplus \stirling{k}{\ell}} + (\uparrow^\ell \downarrow^\ell)^{\oplus \ell \stirling{k}{\ell}} \right)
    \\
    \cong\ \bigoplus_{\ell=1}^{k+1} (\uparrow^\ell \downarrow^{\ell})^{\oplus \left( \ell \stirling{k}{\ell} + \stirling{k}{\ell-1} \right)}
    \cong\ \bigoplus_{\ell=1}^{k+1} (\uparrow^\ell \downarrow^{\ell})^{\oplus \stirling{k+1}{\ell}}. \qedhere
  \end{multline*}
\end{proof}

Recall that, under \cref{twostep}, for each Young diagram $\lambda$ of size $k$, we have the idempotent
\[
  d(e_\lambda) \in \End_\Heis(\uparrow^k \downarrow^k),
\]
where $d$ is the map \cref{diagonal}.  Recall also the definition $P_k(t) = \End_{\Par(t)}(k)$ of the partition algebra.  Let
\[
  \xi =
  \begin{tikzpicture}[anchorbase]
    \pd{0,0};
    \pd{0.5,0};
    \pd{0,0.5};
    \pd{0.5,0.5};
    \draw (0,0) to (0.5,0) to (0.5,0.5) to (0,0.5) to (0,0);
  \end{tikzpicture}
  \in P_2(t)
  \quad \text{and} \quad
  \xi_i = 1_{k-i-1} \otimes \xi \otimes 1_{i-1} \in P_k(t)
  \quad \text{for } 1 \le i \le k-1.
\]
It is straightforward to verify that the intersection of $P_k(t)$ with the tensor ideal of $\Par(t)$ generated by $\xi$ is equal to the ideal $(\xi_1)$ of $P_k(t)$ generated by $\xi_1$.  Denote this ideal by $P_k^\xi(t)$.

As noted in \cite[Lem.~3.1(2)]{CO11}, we have an isomorphism
\begin{equation} \label{dinobot}
  P_k(t)/ P_k^\xi(t) \cong \kk S_k,\quad a + P_k^\xi(t) \mapsto a,\quad a \in \kk S_k,
\end{equation}
where we view elements of $\kk S_k$ as elements of $P_k(t)$ via the homomorphism \cref{garage}.  This observation allows one to classify the primitive idempotents in $P_k(t)$ by induction on $k$.  This classification was first given by Martin in \cite{Mar96}.

\begin{prop} \label{prims}
  For $k > 0$, the primitive idempotents in $P_k(t)$, up to conjugation, are in bijection with the set of Young diagrams $\lambda \in \cY$ with $0 < |\lambda| \le k$.  Furthermore:
  \begin{enumerate}
    \item Under this bijection, idempotents lying in $P_k^\xi(t)$ correspond to Young diagrams $\lambda$ with $0 < |\lambda| < k$.

    \item For each Young diagram $\lambda$ of size $k$, we can choose a primitive idempotent $f_\lambda \in P_k(t)$ corresponding to $\lambda$ so that $f_\lambda + P_k^\xi(t)$ maps to the Young symmetrizer $e_\lambda$ under the isomorphism \cref{dinobot}.
  \end{enumerate}
\end{prop}

\begin{proof}
  This follows as in the proof of \cite[Th.~3.4]{CO11}.  Note that since $t$ is generic,
  \[
    \eta \circ \varepsilon
    =\
    \begin{tikzpicture}[anchorbase]
      \pd{0,0};
      \pd{0,0.5};
    \end{tikzpicture}
  \]
  is not an idempotent in $P_1(t)$.  Thus, the argument proceeds as in the $t=0$ case in \cite[Th.~3.4]{CO11}.
\end{proof}

For $\lambda \in \cY$, define the indecomposable object of $\Rep(S_t)$
\[
  L(\lambda) := (|\lambda|,f_\lambda).
\]

\begin{prop} \label{indecs}
  Fix an integer $k \ge 0$.  The map
  \[
    \lambda \mapsto L(\lambda),\quad \lambda \in \cY,\
  \]
  gives a bijection from the set of $\lambda \in \cY$ with $0 \le |\lambda| \le k$ to the set of nonzero indecomposable objects in $\Rep(S_t)$ of the form $(m,e)$ with $m \le k$, up to isomorphism.  Furthermore
  \begin{enumerate}
    \item If $\lambda \in \cY$ with $0 < |\lambda| \le k$, then there exists an idempotent $e \in P_k(t)$ with $(k,e) \cong L(\lambda)$.

    \item We have that $(0,1_0)$ is the unique object of the form $(m,e)$ that is isomorphic to $L(\varnothing)$.
  \end{enumerate}
\end{prop}

\begin{proof}
  This follows as in the proof of \cite[Lem.~3.6]{CO11}.  Again, our assumption that $t$ is generic implies that we proceed as in the $t=0$ case of \cite[Lem.~3.6]{CO11}.
\end{proof}

We are now ready to prove the main result of this section.

\begin{theo} \label{finally}
  The homomorphism $[\Psi_t]$ of \cref{juice} is injective and its image is
  \begin{equation}
    [\Psi_t] \big( K_0(\Rep(S_t)) \big) = \KDelt(\Sym) \subseteq \rHeis,
  \end{equation}
  where $\rHeis$ is identified with $K_0(\Heis)$ as in \cref{K0Heis}.
\end{theo}

\begin{proof}
  For $k \in \N$, let $\Rep_k(S_t)$ denote the full subcategory of $\Rep(S_t)$ containing the objects of the form $(k,e)$.  By \cref{indecs}, $\Rep_k(S_t)$ is also the full subcategory of $\Rep(S_t)$ containing the objects of the form $(m,e)$, $m \le k$.

  We prove by induction on $k$ that the restriction of $[\Psi_t]$ to $K_0(\Rep_k(S_t))$ is injective, and that
  \[
    [\Psi_t](K_0(\Rep_k(S_t))) = \KDelt (\Sym_{\le k}),
  \]
  where $\Sym_{\le k}$ denotes the subspace of $\Sym$ spanned by symmetric functions of degree $\le k$.  The case $k=0$ is immediate.

  Suppose $k \ge 1$.  The components $\uparrow^k \downarrow^k \to (\uparrow \downarrow)^k$ and $(\uparrow \downarrow)^k \to\ \uparrow^k \downarrow^k$ of the isomorphism \cref{jungle} are
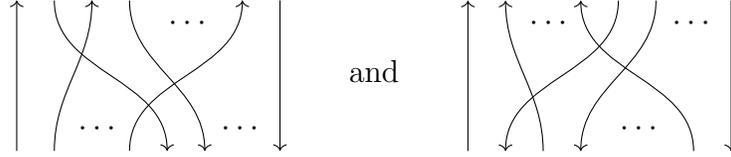
    \[
    \begin{tikzpicture}[anchorbase]
      \draw[->] (0,0) to (0,2);
      \draw[->] (0.5,0) \braidto (1,2);
      \node at (1.1,0.3) {$\cdots$};
      \draw[->] (1.5,0) \braidto (3,2);
      \draw[<-] (2,0) \braidto (0.5,2);
      \draw[<-] (2.5,0) \braidto (1.5,2);
      \node at (3,0.3) {$\cdots$};
      \draw[<-] (3.5,0) \braidto (3.5,2);
      \node at (2.3,1.7) {$\cdots$};
    \end{tikzpicture}
    \qquad \text{and} \qquad
    \begin{tikzpicture}[anchorbase]
      \draw[<-] (0,0) to (0,-2);
      \draw[<-] (0.5,0) to[out=down,in=up] (1,-2);
      \node at (1.1,-0.3) {$\cdots$};
      \draw[<-] (1.5,0) to[out=down,in=up] (3,-2);
      \draw[->] (2,0) to[out=down,in=up] (0.5,-2);
      \draw[->] (2.5,0) to[out=down,in=up] (1.5,-2);
      \node at (3,-0.3) {$\cdots$};
      \draw[->] (3.5,0) to[out=down,in=up] (3.5,-2);
      \node at (2.3,-1.7) {$\cdots$};
    \end{tikzpicture}
  \]
  respectively.  For $i=1,\dotsc,k-1$, consider the morphism (we use \cref{square} here)
  \begin{equation} \label{cadet}
    \Psi_t(s_i - \xi_i)
    =
    \begin{tikzpicture}[anchorbase]
      \draw[->] (0,0) to (0,1);
      \draw[<-] (0.5,0) to (0.5,1);
      \node at (1.25,0.5) {$\cdots$};
      \draw[->] (2,0) to (2,1);
      \draw[<-] (2.5,0) to (2.5,1);
      \draw[->] (3,0) \braidto (4,1);
      \draw[<-] (3.5,0) \braidto (4.5,1);
      \draw[->] (4,0) \braidto (3,1);
      \draw[<-] (4.5,0) \braidto (3.5,1);
      \draw[->] (5,0) to (5,1);
      \draw[<-] (5.5,0) to (5.5,1);
      \node at (6.25,0.5) {$\cdots$};
      \draw[->] (7,0) to (7,1);
      \draw[<-] (7.5,0) to (7.5,1);
    \end{tikzpicture}
    \ \in \End_\Heis \big( (\uparrow \downarrow)^k \big).
  \end{equation}
  Under the isomorphism \cref{jungle}, this corresponds to
  
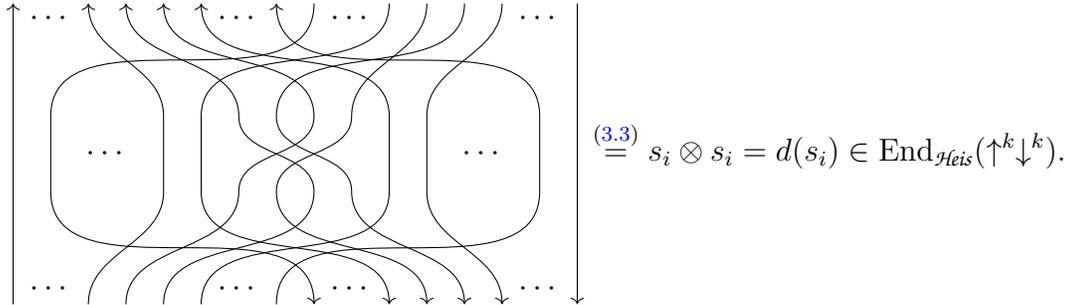
\[
    \begin{tikzpicture}[anchorbase]
      \draw[->] (0,-2) to (0,2);
      \draw[->] (1,-2) \braidto (2,-0.5) to (2,0.5) \braidto (1,2);
      \draw[->] (1.5,-2) \braidto (3,-0.5) \braidto (4,0.5) \braidto (2,2);
      \draw[->] (2,-2) \braidto (4,-0.5) \braidto (3,0.5) \braidto (1.5,2);
      \draw[->] (2.5,-2) \braidto (5,-0.5) to (5,0.5) \braidto (2.5,2);
      \draw[->] (3.5,-2) \braidto (7,-0.5) to (7,0.5) \braidto (3.5,2);
      \draw[<-] (4,-2) \braidto (0.5,-0.5) to (0.5,0.5) \braidto (4,2);
      \draw[<-] (5,-2) \braidto (2.5,-0.5) to (2.5,0.5) \braidto (5,2);
      \draw[<-] (5.5,-2) \braidto (3.5,-0.5) \braidto (4.5,0.5) \braidto (5.5,2);
      \draw[<-] (6,-2) \braidto (4.5,-0.5) \braidto (3.5,0.5) \braidto (6,2);
      \draw[<-] (6.5,-2) \braidto (5.5,-0.5) to (5.5,0.5) \braidto (6.5,2);
      \draw[<-] (7.5,-2) to (7.5,2);
      \node at (0.5,-1.8) {$\cdots$};
      \node at (3,-1.8) {$\cdots$};
      \node at (4.5,-1.8) {$\cdots$};
      \node at (7,-1.8) {$\cdots$};
      \node at (1.25,0) {$\cdots$};
      \node at (6.25,0) {$\cdots$};
      \node at (0.5,1.8) {$\cdots$};
      \node at (3,1.8) {$\cdots$};
      \node at (4.5,1.8) {$\cdots$};
      \node at (7,1.8) {$\cdots$};
    \end{tikzpicture}
    \overset{\cref{H3}}{=} s_i \otimes s_i = d(s_i) \in \End_\Heis(\uparrow^k \downarrow^k).
  \]
  It follows that, for any Young diagram $\lambda$ of size $k$, we have $\Psi_t(e_\lambda) - d(e_\lambda) \in \Psi_t(P_k^\xi(t))$.  Since $f_\lambda - e_\lambda \in P_k^\xi(t)$ by \cref{prims}, this implies that
  \[
    \Psi_t(f_\lambda) - d(e_\lambda)
    = \left( \Psi_t(f_\lambda) - \Psi_t(e_\lambda) \right) +  \left( \Psi_t(e_\lambda) - d(e_\lambda) \right)
    \in \Psi_t(P_k^\xi(t)).
  \]
  Thus, by \cref{prims} and the induction hypothesis, we have
  \[
    [\Psi_t(L(\lambda))] - \KDelt(s_\lambda)
    = [\Psi_t(L(\lambda))] - [\uparrow^k \downarrow^k, d(e_\lambda)]
    \in \KDelt(\Sym_{\le (k-1)}).
  \]
  Since the $s_\lambda$ with $|\lambda| = k$ span the space of degree $k$ symmetric functions, we are done.
\end{proof}

As an immediate corollary of \cref{finally}, we recover the following result of \cite[Cor.~5.12]{Del07}.  (The $T_n$ of \cite{Del07} correspond to the complete symmetric functions.)

\begin{cor} \label{K0P}
  We have an isomorphism of rings $K_0(\Rep(S_t)) \cong \Sym$.
\end{cor}

The Grothendieck ring is one method of decategorification.  Another is the trace, or zeroth Hochschild homology.  We refer the reader to \cite{BGHL14} for details.  The functor $\Psi_t$ induces a ring homomorphism on traces.  We conclude with a brief discussion of this induced map.  First, note that the trace of a category is isomorphic to the trace of its additive Karoubi envelope.  (See \cite[Prop.~3.2]{BGHL14}.)  Thus, $\Tr(\Par(t)) \cong \Tr(\Rep(S_t))$.  In addition, our assumption that $t$ is generic (in particular, $t \notin \N$) implies that $\Rep(S_t)$ is semisimple.  (See \cite[Th.~2.18]{Del07}.)  It follows that the Chern character map
\[
  h \colon K_0(\Rep(S_t)) \to \Tr(\Rep(S_t))
\]
is an isomorphism.  (See \cite[Prop.~5.4]{Sav18}.)  Hence $\Tr(\Par(t)) \cong \Tr(\Rep(S_t)) \cong \Sym$ by \cref{K0P}.  On the other hand, the trace of the Heisenberg category was computed in \cite[Th.~1]{CLLS18} and shown to be equal to a quotient of the W-algebra $W_{1+\infty}$ by a certain ideal $I$.  This quotient contains the Heisenberg algebra $\rHeis$ and the Chern character map induces an injective ring homomorphism
\[
  \rHeis \cong K_0(\Heis) \to \Tr(\Heis) \cong W_{1+\infty}/I.
\]
It follows that the functor $\Psi_t$ induces an injective ring homomorphism
\[
  \Sym \cong \Tr(\Rep(S_t)) \to \Tr(\Heis) \cong W_{1+\infty}/I,
\]
and the image of this map is $\KDelt(\Sym) \subseteq \rHeis \subseteq W_{1+\infty}/I$.

\appendix
\section{Faithfulness over any commutative ring \\ (with Christopher Ryba)\label{appendix}}

In this appendix we prove \cref{faithful} over an arbitrary commutative ring $\kk$.

We say a partition diagram is a \emph{permutation} if it is the image of an element of $S_k$, $k \in \N$, under the map \cref{garage}.   We say a partition diagram is \emph{tensor-planar} if it is a tensor product (horizontal juxtaposition) of partition diagrams consisting of a single block.  Note that every tensor-planar partition diagram is planar (i.e.\ can be represented as a graph without edge crossings inside of the rectangle formed by its vertices) but the converse is false.

Every partition diagram $D$ can be factored as a product $D = D_1 \circ D_2 \circ D_3$, where $D_1$ and $D_3$ are permutations and $D_2$ is tensor-planar.  Furthermore, we may assume that $D_1$ and $D_3$ are compositions of simple transpositions that only transpose vertices in different blocks (since transposing vertices in the same block has no effect).  The number of blocks in $D$ is clearly equal to the number of blocks in $D_2$.  For example, the partition diagram
\[
  D =
  \begin{tikzpicture}[anchorbase]
    \pd{0.5,0};
    \pd{1,0};
    \pd{1.5,0};
    \pd{2,0};
    \pd{0,0.5};
    \pd{0.5,0.5};
    \pd{1,0.5};
    \pd{1.5,0.5};
    \pd{2,0.5};
    \draw (0,0.5) to[out=down,in=down] (1,0.5);
    \draw (0.5,0) to[out=up,in=up,looseness=0.5] (1,0);
    \draw (0.5,0.5) to[out=down,in=down,looseness=0.5] (1.5,0.5);
    \draw (1,0) \braidto (2,0.5);
    \draw (2,0) \braidto (1.5,0.5);
  \end{tikzpicture}
\]
has four blocks and decomposition $D = D_1 \circ D_2 \circ D_3$, where
\[
  D_1 =
  \begin{tikzpicture}[anchorbase]
    \pd{0,0};
    \pd{0.5,0};
    \pd{1,0};
    \pd{1.5,0};
    \pd{2,0};
    \pd{0,0.5};
    \pd{0.5,0.5};
    \pd{1,0.5};
    \pd{1.5,0.5};
    \pd{2,0.5};
    \draw (0,0) \braidto (0,0.5);
    \draw (1.5,0) \braidto (1.5,0.5);
    \draw (2,0) \braidto (2,0.5);
    \draw (0.5,0) \braidto (1,0.5);
    \draw (1,0) \braidto (0.5,0.5);
  \end{tikzpicture},
  \quad
  D_2 =
  \begin{tikzpicture}[anchorbase]
    \pd{0,0};
    \pd{0.5,0};
    \pd{1,0};
    \pd{1.5,0};
    \pd{0,0.5};
    \pd{0.5,0.5};
    \pd{1,0.5};
    \pd{1.5,0.5};
    \pd{2,0.5};
    \draw (0,0.5) to[out=down,in=down] (0.5,0.5);
    \draw (0.5,0) \braidto (1,0.5) to[out=down,in=down] (1.5,0.5);
    \draw (1,0) to[out=up,in=up] (1.5,0) \braidto (2,0.5);
  \end{tikzpicture}
  =
  \begin{tikzpicture}[{>=To,baseline={(0,0.15)}}]
    \pd{0,0};
    \draw (0,0) to (0,0.25);
    \node at (0,.5) {};
  \end{tikzpicture}
  \otimes
  \begin{tikzpicture}[{>=To,baseline={(0,0.15)}}]
    \pd{0,0.5};
    \pd{0.5,0.5};
    \draw (0,0.5) to[out=down,in=down,looseness=1.5] (0.5,0.5);
    \node at (0,0) {};
  \end{tikzpicture}
  \otimes
  \begin{tikzpicture}[{>=To,baseline={(0,0.15)}}]
    \pd{0,0};
    \pd{0,0.5};
    \pd{0.5,0.5};
    \draw (0,0) to (0,0.5) to[out=down,in=down,looseness=1.5] (0.5,0.5);
  \end{tikzpicture}
  \otimes
  \begin{tikzpicture}[{>=To,baseline={(0,0.15)}}]
    \pd{0,0};
    \pd{0.5,0};
    \pd{0,0.5};
    \draw (0.5,0) to[out=up,in=up,looseness=1.5] (0,0) to (0,0.5);
  \end{tikzpicture}
  ,
  \quad
  D_3 =
  \begin{tikzpicture}[anchorbase]
    \pd{0.5,0};
    \pd{1,0};
    \pd{1.5,0};
    \pd{2,0};
    \pd{0.5,0.5};
    \pd{1,0.5};
    \pd{1.5,0.5};
    \pd{2,0.5};
    \draw (0.5,0) to[out=up,in=down,looseness=0.5] (1.5,0.5);
    \draw (1,0) to[out=up,in=down,looseness=0.5] (2,0.5);
    \draw (1.5,0) to[out=up,in=down,looseness=0.5] (0.5,0.5);
    \draw (2,0) to[out=up,in=down,looseness=0.5] (1,0.5);
  \end{tikzpicture}.
\]

For $n,k,\ell \in \N$, let $\Hom_{\Par}^{\le n}(k,\ell)$ denote the subspace of $\Hom_{\Par}(k,\ell)$ spanned by partition diagrams with at most $n$ blocks.  Composition respects the corresponding filtration on morphism spaces.

Recall the bases of the morphism spaces of $\Heis$ given in \cite[Prop.~5]{Kho14}.  For any such basis element $X$ in $\Hom_{\Peis}\big( (\uparrow \downarrow)^k, (\uparrow \downarrow)^\ell \big)$, define the \emph{block number} of $X$ to be number of distinct closed (possibly intersecting) loops in the diagram
\[
  \begin{tikzpicture}[anchorbase]
    \draw[->] (0,0) -- (0,0.1) arc (180:0:.25) -- (0.5,0);
  \end{tikzpicture}^{\otimes \ell}
  \circ
  X
  \circ
  \begin{tikzpicture}[anchorbase]
    \draw[<-] (0,1) -- (0,0.9) arc (180:360:.25) -- (0.5,1);
  \end{tikzpicture}^{\otimes k}.
\]
For $n \in \N$, let $\Hom_{\Peis}^{\le n} \big( (\uparrow \downarrow)^k, (\uparrow \downarrow)^\ell \big)$ denote the subspace of $\Hom_{\Peis} \big( (\uparrow \downarrow)^k, (\uparrow \downarrow)^\ell \big)$ spanned by basis elements with block number at most $n$.  Composition respects the resulting filtration on morphism spaces.

The image under $\Psi_t$ of tensor-planar partition diagrams (writing the image in terms of the aforementioned bases of the morphism spaces of $\Heis$) is particularly simple to describe.  Since each tensor-planar partition diagram is a tensor product of single blocks, consider the case of a single block.  Then, for example, we have
\[
  \Psi_t
  \left(
    \begin{tikzpicture}[anchorbase]
      \pd{0,0.5};
      \pd{0.5,0.5};
      \pd{1,0.5};
      \draw (0,0.5) to[out=up,in=up,looseness=1.5] (0.5,0.5);
      \draw (0.5,0.5) to[out=up,in=up,looseness=1.5] (1,0.5);
    \end{tikzpicture}
  \right)
  =
  \begin{tikzpicture}[>=To,baseline={(0,0.2)}]
    \draw[<-] (1.5,0) to[out=up,in=up,looseness=2] (2,0);
    \draw[<-] (0.5,0) to[out=up,in=up,looseness=2] (1,0);
    \draw[->] (0,0) to[out=up,in=up] (2.5,0);
  \end{tikzpicture}
  \quad \text{and} \quad
  \Psi_t
  \left(
    \begin{tikzpicture}[anchorbase]
      \pd{0,0};
      \pd{0.5,0};
      \pd{1,0};
      \pd{0,0.5};
      \pd{0.5,0.5};
      \draw (1,0) to (0,0) to (0,0.5) to (0.5,0.5);
    \end{tikzpicture}
  \right)
  =
  \begin{tikzpicture}[anchorbase]
    \draw[->] (0,0) \braidto (0.5,1);
    \draw[<-] (0.5,0) to[out=up,in=up,looseness=2] (1,0);
    \draw[<-] (1.5,0) to[out=up,in=up,looseness=2] (2,0);
    \draw[<-] (2.5,0) \braidto(2,1);
    \draw[->] (1,1) to[out=down,in=down,looseness=2] (1.5,1);
  \end{tikzpicture}
  \ .
\]
The general case is analogous.  (In fact, the images of all planar partition diagrams are similarly easy to describe.)  In particular, if $D$ is a tensor-planar partition diagram with $n$ blocks, then $\Psi_t(D)$ is a planar diagram with block number $n$.

For a permutation partition diagram $D \colon k \to k$, let $T(D)$ be the planar diagram (a morphism in $\Peis$) defined as follows: Write $D = s_{i_1} \circ s_{i_2} \circ \dotsb \circ s_{i_r}$ as a reduced word in simple transpositions and let
\[
  T(D) = \Psi_t(s_{i_1} - \xi_{i_1}) \circ \Psi_t(s_{i_2} - \xi_{i_2}) \circ \dotsb \circ \Psi_t(s_{i_r} - \xi_{i_r}).
\]
(See \cref{cadet}.)  It follows from the braid relations \cref{genbraid} that $T(D)$ is independent of the choice of reduced word for $D$.

\begin{prop} \label{locker}
  Suppose $D \colon k \to \ell$ is a partition diagram with $n$ blocks.  Write $D = D_1 \circ D_2 \circ D_3$, where $D_2$ is a tensor-planar partition diagram and $D_1$ and $D_3$ are compositions of simple transpositions that only transpose vertices in different blocks.  Then
  \[
    \Psi_t(D) - T(D_1) \circ \Psi_t(D_2) \circ T(D_3) \in \Hom_{\Peis}^{\le n-1} \big( (\uparrow \downarrow)^k, (\uparrow \downarrow)^\ell \big).
  \]
\end{prop}

\begin{proof}
  We have $\Psi_t(D) = \Psi_t(D_1) \circ \Psi_t(D_2) \circ \Psi_t(D_3)$.  As noted above, $D_2$ has $n$ blocks and $\Psi_t(D_2)$ has block number $n$.  Suppose $1 \le j < \ell$.  If $D' \colon k \to \ell$ is a partition diagram with $n$ blocks such that $j'$ and $(j+1)'$ lie in different blocks, then $\xi_j \circ D'$ has $n-1$ blocks.  It follows that
  \[
    \Psi_t(s_j) \circ \Psi_t(D') - \Psi_t(s_j - \xi_j) \circ \Psi_t(D')
    \in \Hom_{\Peis}^{\le n-1} \big( (\uparrow \downarrow)^k, (\uparrow \downarrow)^\ell \big).
  \]
  Similarly,
  \[
    \Psi_t(D') \circ \Psi_t(s_j) - \Psi_t(D') \circ \Psi_t(s_j - \xi_j)
    \in \Hom_{\Peis}^{\le n-1} \big( (\uparrow \downarrow)^k, (\uparrow \downarrow)^\ell \big)
  \]
  for any $1 \le j < k$ such that $j$ and $j+1$ lie in different blocks of $D'$.  The result then follows by writing $D_1$ and $D_3$ as reduced words in simple transpositions.
\end{proof}

\begin{cor}
  The functor $\Psi_t$ is faithful over an arbitrary commutative ring $\kk$.
\end{cor}

\begin{proof}
  It is clear that, in the setting of \cref{locker}, $T(D_1) \circ \Psi_t(D_2) \circ T(D_3)$ is uniquely determined by $D$.  Indeed, $D$ is the partition diagram obtained from $T(D_1) \circ \Psi_t(D_2) \circ T(D_3)$ by replacing each pair $\uparrow \downarrow$ by a vertex and each strand by an edge.  Furthermore, the diagrams of the form $T(D_1) \circ \Psi_t(D_2) \circ T(D_3)$ are linearly independent by \cite[Prop.~5]{Kho14}.  The result then follows by a standard triangularity argument.
\end{proof}


\bibliographystyle{alphaurl}
\bibliography{DeligneHeisenberg}

\end{document}